\documentclass{article}
\usepackage{amsmath,amscd,amstext,amssymb,amsfonts}
\usepackage{amsthm}
\usepackage{graphics}
\usepackage{latexsym,empheq}
\usepackage{color}
\usepackage{bookman}
\usepackage{bm}
\usepackage{cite}
\usepackage{enumerate}

\newcommand{\nd}{\ensuremath {{d}}} %%% dimension variable in this file: d

\newcommand{\beq}{\begin{eqnarray}}
\newcommand{\eeq}{\end{eqnarray}}
\newcommand{\beqq}{\begin{eqnarray*}}
\newcommand{\eeqq}{\end{eqnarray*}}
\newcommand{\bit}{\begin{itemize}}
\newcommand{\eit}{\end{itemize}}
\newcommand{\benu}{\begin{enumerate}}
\newcommand{\eenu}{\end{enumerate}}
\newcommand{\ds}{\displaystyle}
\newcommand{\R}{\ensuremath{\mathbb R}}\newcommand{\real}{\R}
\newcommand{\Rn}{\ensuremath{{\mathbb R}^{\nd}}} \newcommand{\rn}{\Rn}
\newcommand{\N}{{\ensuremath{\mathbb N}}}\newcommand{\nat}{\N}
\newcommand{\No}{\ensuremath{\N_{0}}}\newcommand{\no}{\N_0}
\newcommand{\Z}{\mathbb Z} 
\newcommand{\Zn}{\Z^{\nd}}
\newcommand{\C}{{\mathbb C}}
\newcommand{\SRn}{\mathcal{S}(\Rn)}
\newcommand{\SpRn}{\mathcal{S}'(\Rn)}
\newcommand{\cD}{\mathcal{D}}

\newcommand{\dint}{\;\mathrm{d}}
\newcommand{\Ft}{\mathcal{F}}
\newcommand{\Fti}{\mathcal{F}^{-1}}
\def\dist{\mathop{\rm dist}\nolimits}    %NEW

\newcommand{\A}{\ensuremath{A^s_{p,q}}}  
\newcommand{\Ae}{\ensuremath{A^{s_1}_{p_1,q_1}}}  
\newcommand{\Az}{\ensuremath{A^{s_2}_{p_2,q_2}}}  
\newcommand{\B}{\ensuremath{B^s_{p,q}}}  
\newcommand{\be}{\ensuremath{B^{s_1}_{p_1,q_1}}}  
\newcommand{\bz}{\ensuremath{B^{s_2}_{p_2,q_2}}}  
\newcommand{\F}{\ensuremath{F^s_{p,q}}}

\def\supp{\mathop{\rm supp}\nolimits}
\newcommand{\bpr}{\begin{proof}}
\newcommand{\epr}{\end{proof}}
\newcommand{\bli}{\begin{list}{}{\labelwidth6mm\leftmargin8mm}}
\newcommand{\eli}{\end{list}}

\def\id{\mathop{\rm id}\nolimits}

\newcommand{\tr}{\ensuremath{\mathop{\mathrm{tr}}\nolimits}}
\newcommand{\ignore}[1]{}
\newcommand{\nn}[1]{\ensuremath \nu( #1)}
\newcommand{\tn}{\ensuremath\mathbf{t}}
\allowdisplaybreaks[1]
\numberwithin{equation}{section}

\newtheorem{lemma}{Lemma}[section]
\newtheorem{corollary}[lemma]{Corollary}
\newtheorem{proposition}[lemma]{Proposition}
\newtheorem{theorem}[lemma]{Theorem}

\theoremstyle{definition}
\newtheorem{definition}[lemma]{Definition}
\newtheorem{example}[lemma]{Example}

\theoremstyle{remark}
\newtheorem{remark}[lemma]{Remark}

\title{Nuclear embeddings in general vector-valued sequence spaces with an application to Sobolev embeddings of function spaces on quasi-bounded domains}
\author{Dorothee D. Haroske\footnotemark[1], \ Hans-Gerd Leopold \ and Leszek Skrzypczak\footnotemark[1]\ \footnotemark[2]}
\date{\today}

\begin{document}

\maketitle

\footnotetext[1]{Both authors were partially supported by the German Research Foundation (DFG), Grant no. Ha 2794/8-1.}
\footnotetext[2]{The author was supported by National Science Center, Poland,  Grant no. 2013/10/A/ST1/00091.}

%\ignore{
\begin{abstract}
We study nuclear embeddings for spaces of Besov and Triebel-Lizorkin
  type defined on quasi-bounded domains $\Omega\subset\Rn$. The counterpart for such function spaces defined on bounded domains has been considered for a long time and the complete answer was obtained only recently. Compact embeddings for function spaces defined on quasi-bounded domains have already been studied in detail, also concerning their entropy and $s$-numbers. We now prove the first and complete nuclearity result in this context. The concept of nuclearity has  already been introduced by Grothendieck in 1955.
  Our second main contribution is the generalisation of the famous Tong result (1969) which characterises nuclear diagonal operators acting between sequence spaces of $\ell_r$ type, $1\leq r\leq\infty$. We can now extend this to the setting of general vector-valued sequence spaces of type $\ell_q(\beta_j \ell_p^{M_j})$ with $1\leq p,q\leq\infty$, $M_j\in\no$ and a weight sequence with $\beta_j>0$. In particular, we prove a criterion for the embedding $\id_\beta : \ell_{q_1}(\beta_j \ell_{p_1}^{M_j}) \hookrightarrow \ell_{q_2}(\ell_{p_2}^{M_j})$ to be nuclear.\\

\noindent  {\em Keywords:}~ nuclear embeddings, vector-valued sequence spaces, function spaces on quasi-bounded domains \\
  {\em MSC (2010):}~46E35, 47B10
  
\end{abstract}
%}

\section*{Introduction}

  Grothendieck introduced the concept of nuclearity in \cite{grothendieck} more than 60 years ago. It provided the basis for many famous developments in functional analysis afterwards, for instance Enflo used nuclearity in his famous solution \cite{enflo} of the approximation problem, a long-standing problem of Banach from the Scottish Book. We refer to \cite{Pie-snumb,Pie-op-2}, and, in particular, to \cite{pie-history} for further historic details.

Let $X,Y$ be Banach spaces, $T\in \mathcal{L}(X,Y)$ a linear and bounded operator. Then $T$ is called {\em nuclear}, denoted by $T\in\mathcal{N}(X,Y)$, 
if there exist elements $a_j\in X'$, the dual space of $X$, and $y_j\in Y$, $j\in\mathbb{N}$, such that $\sum_{j=1}^\infty \|a_j\|_{X'} \|y_j\|_Y < \infty$ 
and a nuclear representation $Tx=\sum_{j=1}^\infty a_j(x) y_j$ for any $x\in X$. Together with the {\em nuclear norm}
\[
 \nu(T)=\inf\Big\{ \sum_{j=1}^\infty   \|a_j\|_{X'} \|y_j\|_Y:\ T =\sum_{j=1}^\infty a_j(\cdot) y_j\Big\},
  \]
  where the infimum is taken over all nuclear representations of $T$, the space $\mathcal{N}(X,Y)$ becomes a Banach space. It is obvious that 
	nuclear operators are, in particular, compact.
	
  Already in the early years there was a strong interest to study examples of nuclear operators beyond diagonal operators in $\ell_p$ 
	sequence spaces, where a complete answer was obtained in \cite{tong}.
	% (with some partial forerunner in \cite{Pie-op-2}). 
	Concentrating on embedding operators in spaces of Sobolev type, first results can be found, for instance, in {\cite{PiTri}, \cite{Pie-r-nuc}, \cite{Parfe-2}}. We noticed an increased interest in studies of nuclearity in the last years.  Dealing with the Sobolev embedding for spaces on a bounded domain, some of the recent papers we 
have in mind are \cite{EL-4,EGL-3,Tri-nuclear,CoDoKu,CoEdKu} using quite different techniques however. 

There might be several reasons for this. For example, the problem to describe a compact operator outside the Hilbert space setting is a partly
 open and very important one. 
 It is well known from the remarkable Enflo result \cite{enflo} that there are compact operators
 between Banach spaces which cannot be approximated by finite-rank operators.
 This led to a number of -- meanwhile well-established  and famous -- methods to circumvent this difficulty and find alternative ways to `measure' the compactness or `degree' of compactness of an operator. It can be described by the asymptotic behaviour of its approximation or entropy numbers, 
which are basic tools for many different problems nowadays, e.g. eigenvalue distribution of compact operators in Banach spaces, 
optimal approximation of Sobolev-type embeddings, but also for numerical questions. In all these problems, the decomposition
of a given compact operator into a series is an essential proof technique. It turns out that in many of the recent contributions \cite{Tri-nuclear,CoDoKu,CoEdKu} studying nuclearity, a key tool in the arguments are new decomposition techniques as well, adapted to the different spaces. Inspired by the nice paper \cite{CoDoKu} we also used such arguments in our paper \cite{HaSk-nuc-weight}, and intend to follow this strategy here again.

We have two goals in this paper: we want to inquire into the nature of compact embeddings when the function spaces of Besov and Triebel-Lizorkin type are defined on certain unbounded domains. It is well known, that such function spaces defined on $\rn$ never admit a compact, let alone nuclear embedding. But replacing $\rn$ by a bounded Lipschitz domain $\Omega\subset\rn$, then the question of nuclearity has already been solved, cf. \cite{Pie-r-nuc} (with a forerunner in \cite{PiTri}) for the sufficient conditions, and 
the recent paper \cite{Tri-nuclear} with some forerunner in \cite{Pie-r-nuc} and partial results in \cite{EGL-3,EL-4} for the necessity of the conditions. In \cite{HaSk-nuc-weight} we also contributed a little to these questions. More precisely, for Besov spaces on bounded Lipschitz domains, $B^s_{p,q}(\Omega)$, it is well known that
\[
\id_\Omega^B : \be(\Omega) \hookrightarrow \bz(\Omega)\] %\quad
\text{is nuclear\quad if, and only if,} %\quad
\[    s_1-s_2 > \nd-\nd \max\left(\frac{1}{p_2}-\frac{1}{p_1},0\right),
\]
where $1\leq p_i,q_i\leq \infty$, $s_i\in\real$, $i=1,2$. So a natural question appears whether for some unbounded domains compactness, or now nuclearity, of the corresponding embedding is possible. 
This question has already been answered  in the affirmative, concerning special, so-called quasi-bounded domains, and its compactness. In \cite{DZ,Leo-Sk-2, Leo-Sk-3} the above-mentioned `degree' of this compactness has been characterised in terms of the asymptotic behaviour of $s$-numbers and entropy numbers of $\id^B_\Omega$.
Now we study its nuclearity and find almost complete characterisations below, depending on {some} number $b(\Omega)$ which describes the {size} of such an unbounded domain; we refer for all definitions and details to Section~\ref{sect-fs}. We find, in Theorem~\ref{nuclear-quasi} below, that for domains $\Omega$ with $b(\Omega)=\infty$, there is a nuclear embedding if, and only if, we are in the extremal situation $p_1=1$, $p_2=\infty$, and $s_1-s_2>\nd$, while in the case $b(\Omega)<\infty$ the conditions look more similar to the above ones for bounded domains; e.g., for such quasi-bounded domains $\id_\Omega^B$ is nuclear, if
\[
s_1-s_2 -\nd \left(\frac{1}{p_1}-\frac{1}{p_2}\right) > \frac{b(\Omega)}{\tn(p_1,p_2)},
\]
where $1\leq p_i,q_i\leq\infty$, $i=1,2$, $s_1>s_2$, and the new quantity $\tn(p_1,p_2)$ is defined by $\ \frac{1}{\tn(p_1,p_2)} = 1 - \max (\frac{1}{p_1}-\frac{1}{p_2},0)$. For bounded Lipschitz domains $\Omega$ one has $b(\Omega)=d$ and then the above-mentioned previous result is recovered. 
As already indicated, we follow here the general ideas presented in \cite{CoDoKu} which use decomposition techniques and benefit from  Tong's result \cite{tong} about nuclear diagonal operators acting in sequence spaces of type $\ell_p$. In our situation, however, it turned out that we need more general vector-valued sequence spaces of type $\ell_q(\beta_j \ell_p^{M_j})$, {cf. Definition \ref{seq-sp-def1} below}. 
%  which contain  complex sequences $x=(x_{j,m})_{j\in\no, m=1, \dots, M_j}$ such %that
%\[ 
 % \left\|x|\ell_q\left(\beta_j \ell_p^{M_j}\right)\right\| = \Big(\sum_{j=0}^\infty %\beta_j^q \big(\sum_{k=1}^{M_j} |x_{j,k}|^p\Big)^{\frac{q}{p}}\Big)^{\frac1q}
%\]
%is finite (with appropriate modifications if $p=\infty$ or $q=\infty$). Here %$\beta_j>0$ and $M_j\in \no$. 
So we were led to the study of the embedding
\[
\id_\beta : \ell_{q_1}(\beta_j \ell_{p_1}^{M_j}) \hookrightarrow \ell_{q_2}(\ell_{p_2}^{M_j})
\]
in view of its nuclearity. Concerning its compactness this has already been investigated in \cite{Leo-00,Leo-fsdona}, but we did not find the corresponding nuclearity result in the literature and decided to study this question first. We tried to follow Tong's ideas in \cite{tong} as much as possible and could finally prove that for $1\leq  p_i, q_i\leq\infty$, $i=1,2$, 
\[ \id_\beta \quad\text{is nuclear\quad if, and only if, }\quad \left(\beta_j^{-1} M_j^{\frac{1}{\tn(p_1,p_2)}}\right)_{j\in\no} \in \ell_{\tn(q_1,q_2)}.
\]
This is the first main outcome of our paper and, to the best of our knowledge, also the first result in this direction.  In case of $M_j\equiv 1$ this coincides with Tong's result \cite{tong}.

The paper is organised as follows. In Section~\ref{prelim} we recall basic facts 
about the sequence and function spaces we shall work with, Section~\ref{sect-nuc} is devoted to the question of nuclear embeddings in general vector-valued sequence spaces, with our first main outcome in Theorem~\ref{nucl-seq-sp}. In Section~\ref{sect-fs} we return to the setting of function spaces, now appropriately extended to function spaces on certain unbounded, so-called quasi-bounded domains. We are able to prove an almost complete result for the corresponding Besov spaces on quasi-bounded domains in Theorem~\ref{nuclear-quasi}, using appropriate wavelet decompositions and our findings in Section~\ref{sect-nuc}. Finally this also leads to a corresponding result for Triebel-Lizorkin spaces on quasi-bounded domains.

\section{Sequence and function spaces}\label{prelim}
We start to define the sequence spaces and function spaces we are interested in. First of all we need to fix some notation.
By $\N$ we denote the set of natural numbers, 
by $\No$ the set $\N \cup \{ 0\}$, 
and by $\Zn$ the set of all lattice points
in $\Rn$ having integer components.  

Let $a_+ = \max(a,0)$, $a\in\real$. For two positive real sequences 
$(a_k)_{k\in \No}$ and $(b_k)_{k\in \No}$ we mean by 
$a_k\sim b_k$ that there exist constants $c_1,c_2>0$ such that $c_1 a_k\leq
b_k\leq c_2  a_k$ for all $k\in \No$; similarly for positive
functions.   

Given two (quasi-) Banach spaces $X$ and $Y$, we write $X\hookrightarrow Y$
if $X\subset Y$ and the natural embedding of $X$ in $Y$ is continuous. 

All unimportant positive constants will be denoted by $c$, occasionally with
subscripts. For convenience, let both $ \dint x\ $ and $\ |\cdot|\ $ stand for the
($\nd$-dimensional) Lebesgue measure in the sequel.

%%%%%%%%%%%

%%%%%%%%%%%

\subsection{Sequence spaces}
We begin with the definition of weighted vector-valued sequence spaces. We consider a general weight sequence $(\beta_j)_{j=0}^\infty$ of positive real numbers and and a sequence  $(M_j)_{j=0}^\infty$ of positive integers that are dimensions of finite-dimensional vector spaces. 
\begin{definition}\label{seq-sp-def1}
  Let $0<p,q\leq\infty$, $(\beta_j)_{j\in\no}$ be a weight sequence, that is $\beta_j > 0$, and $(M_j)_{j\in\no}$ be a sequence of natural numbers. Then
  \begin{align*}
    \ell_q\left(\beta_j \ell_p^{M_j}\right) =  \Big\{ & x=(x_{j,k})_{j\in\no, k=1, \dots, M_j}\ : \ {x_{j,k}\in\C}, \\  
& \left\|x|\ell_q\left(\beta_j \ell_p^{M_j}\right)\right\| = \Big(\sum_{j=0}^\infty \beta_j^q \big(\sum_{k=1}^{M_j} |x_{j,k}|^p\Big)^{\frac{q}{p}}\Big)^{\frac1q}<\infty
    \Big\},
  \end{align*}
with the usual modifications if $p=\infty$ and/or $q=\infty$.
\end{definition}

\begin{remark} %We collect some more or less immediate observations.
%  \benu[(i)]
%  \item
The spaces $\ell_q(\beta_j \ell_p^{M_j})$ are quasi-Banach spaces (Banach spaces if $p,q\geq 1$). 
%  \item
If $M_j \equiv 1$, then the space $ \ell_q(\beta_j \ell_p^{M_j})$ coincides with the usual weighted space  $\ell_q\left(\beta_j\right)$, that is, for $\beta_j\equiv 1$, nothing else than $\ell_q$.
%\eenu
\end{remark}

We recall necessary and sufficient conditions for  the  compactness of an embedding of  the sequence spaces. Let us introduce the following notation, which will be important for us in the sequel: for 
 $0<p_i,q_i\leq\infty$, $i=1,2$, we define 
\begin{equation}\label{pq-star}
  \frac{1}{p^*} := \left( \frac{1}{p_2} - \frac{1}{p_1}\right)_+\qquad\text{and}\qquad  %\max\left( \frac{1}{p_2} - \frac{1}{p_1},0\right),
 \frac{1}{q^*} := \left( \frac{1}{q_2} - \frac{1}{q_1}\right)_+ %\max\left( \frac{1}{q_2} - \frac{1}{q_1},0\right) 
\end{equation}
(with the understanding that $p^\ast=\infty$ when $p_1\leq p_2$, $q^\ast=\infty$ when $q_1\leq q_2$). Recall that $c_0$ denotes the subspace of $\ell_\infty$ containing the null sequences. 

%%%%%%%%%%%%%%%%%%%%%%%%%%%%%%
  \begin{remark}\label{rem-compact-seq}
    We are especially interested in properties of the embedding
  \begin{equation}\label{id_beta}
	\id_\beta : \ell_{q_1}\left(\beta_j \ell_{p_1}^{M_j}\right) \rightarrow 
	\ell_{q_2}\left(\ell_{p_2}^{M_j}\right). 
\end{equation}
It was proved in \cite{Leo-00,Leo-fsdona} that the embedding is compact if, and only if,
\begin{align}\label{cond-comp-id_beta}
	\ds  \left(\beta_j^{-1} M_j^{\frac{1}{p^\ast}}\right)_{j\in\no} \in \ell_{q^\ast}, 
\end{align}
where for $q^\ast=\infty$ the space $\ell_\infty$ has to be replaced by $c_0$. 	 In Section~\ref{sect-nuc} below we study its nuclearity.
\end{remark}

%%%%%%%%%%%%%%%%%%%%%%%%%

%%%%%%%%%%%%%%%%%%%%%%%%%%%%
%The following proposition  was proved in \cite{Leo-00,Leo-fsdona}.
%\begin{proposition}\label{compact-seq}
%  Let $0<p_1,p_2\leq\infty$, $0<q_1,q_2\leq\infty$, $(\beta_j)_{j\in\no}$ be an %arbitrary weight sequence and $(M_j)_{j\in\no}$ be a sequence of natural numbers. %Then the embedding
%  \begin{equation}\label{id_beta}
%\id_\beta : \ell_{q_1}\left(\beta_j \ell_{p_1}^{M_j}\right) \rightarrow 
%\ell_{q_2}\left(\ell_{p_2}^{M_j}\right)
%  \end{equation}
%  is compact if, and only if,
%\begin{align}\label{cond-comp-id_beta}
% 	 \ds  \left(\beta_j^{-1} M_j^{\frac{1}{p^\ast}}\right)_{j\in\no} \in %\ell_{q^\ast}, 
% 	\end{align}
%  where for $q^\ast=\infty$ the space $\ell_\infty$ has to be replaced by $c_0$. 	 
%\end{proposition}
%
%We refer to the papers \cite{Leo-00,Leo-fsdona} for further results on the %continuity of the above embedding $\id_\beta$ given by \eqref{id_beta}, in %Section~\ref{sect-nuc} below we study its nuclearity.
%%%%%%%%%%%%%%%%%%%%%%%%%%%%%%%%%%

\subsection{Function spaces and compact Sobolev embeddings}
\label{sect-1}

We now consider some function spaces. Let $\ 0<p\leq \infty$. Then the Lebesgue space $L_p(\Rn)$ contains all
measurable functions such that  
\begin{equation}
\|f\;|L_p(\Rn)\|= \left( \int_{\Rn}|f(x)|^p  \dint x\right)^{1/p},\quad 0<p<\infty,
\end{equation}
is finite, where for $p=\infty$ this is the classical Lebesgue space of measurable, essentially bounded functions, $L_\infty(\Rn)$.

The Schwartz space $\ \SRn\ $ and its dual $\ \SpRn\ $ of all 
complex-valued tempered distributions have their usual meaning here.
Let  $\ \varphi_0=\varphi \in \SRn\ $ be such that  
\begin{equation}
\supp \varphi\subset\left\{y\in\Rn:|y|<2\right\}\quad \mbox{and}\quad
\varphi(x)=1\quad\mbox{if}\quad |x|\leq 1\;,
\end{equation}
and for each $\ j\in\N\;$ let $\ \varphi_j(x)=
\varphi(2^{-j}x)-\varphi(2^{-j+1}x)$. Then $\ (\varphi_j)_{j=0}^\infty $
forms a {\em smooth dyadic resolution of unity}. Given any $\ f\in \SpRn$, we
denote by $\Ft f$ and $\Fti f$ its Fourier transform and its
inverse Fourier transform, respectively.
%Let $f\in \SpRn$, then
%the compact support of $\varphi_j \Ft f$ implies by the Paley-Wiener-Schwartz
%theorem that $\Fti(\varphi_j \Ft f)$ is an entire analytic 
%function on $\Rn$.

\begin{definition}
Let $0 <p,q\leq \infty$, $\ s\in \mathbb{R}$ and $\ \left(\varphi_j\right)_j$
a smooth dyadic resolution of unity. 
\benu[\bfseries\upshape (i)]
\item
  {\em The Besov  space}
  $B_{p,q}^{s}(\rn)$ is the set of all distributions $f\in \SpRn$ such that 
\begin{align}\label{B}
\big\|f\;|B_{p,q}^{s}(\rn)\big\|=
\left\| \left( 2^{js}\big\|\mathcal{F}^{-1}(\varphi_j\mathcal{F}f)|
L_p(\rn)\big\|\right)_{j\in\no} | \ell_q\right\| 
\end{align}
is finite. 
\item Assume $0<p<\infty$. {\em The Triebel-Lizorkin  space} $F_{p,q}^{s}(\rn)$  
is the set of all distributions $f\in   \SpRn$ such that 
\begin{align}\label{Fw}
\big\|f\;|F_{p,q}^{s}(\rn)\big\|=\left\| \big\| \left(2^{js}|
 \mathcal{F}^{-1}(\varphi_j\mathcal{F}f)(\cdot)|\right)_{j\in\no} |
 \ell_q\big\|~ |L_p(\rn)\right\| 
\end{align} 
is finite. 
\eenu
\end{definition}

\begin{remark}
The  spaces $B_{p,q}^s(\rn)$ and $F_{p,q}^s(\rn)$ are
independent of the particular choice of the smooth dyadic resolution of unity
$\left(\varphi_j\right)_j $ appearing in their definitions. They are quasi-Banach spaces
(Banach spaces for $p,q\geq 1$), and $\ \SRn \hookrightarrow \B(\Rn)\hookrightarrow \SpRn$, similarly for the $F$-case, where the
first embedding is dense if $p,q<\infty$; we
refer, in particular, to the series of monographs by {Triebel} 
\cite{T-F1,T-F2,T-Frac,T-F3} for a 
comprehensive treatment of these spaces. \\
 Concerning (classical) Sobolev spaces $W^k_p(\Rn) $ built
upon $L_p(\Rn)$ in the usual way, it holds 
\beq
W^k_p(\Rn) = F^k_{p,2}(\Rn), \quad k\in\No, \quad 1<p<\infty.
\label{W=F}
\eeq
\end{remark}

{\em Convention}. We adopt the nowadays usual custom to write $\A$ instead of $\B$ or $\F$, respectively,
when both scales of spaces are meant simultaneously in some context (but
always with the understanding of the same choice within one and the same
embedding, if not otherwise stated explicitly).

%\ignore{
\begin{remark}\label{A-emb}
  Occasionally we use the following elementary embeddings. %which are natural extensions from the unweighted case.
  If $0<q\leq\infty$,  $0<q_0\leq q_1\leq\infty$, $\ 0<p<\infty$, $\ s, s_0,s_1\in \mathbb{R}$ with $s_1\leq s_0$, %and $\ w\in\mathcal{A}_\infty$,
  then $A^{s_0}_{p,q}(\Rn) \hookrightarrow
A^{s_1}_{p,q}(\Rn)$, $A^{s}_{p,q_0}(\Rn) \hookrightarrow A^{s}_{p,q_1}(\Rn)$, and
\begin{equation}\label{B-F-B}
B^s_{p, \min(p,q)}(\Rn) \hookrightarrow \F(\Rn) \hookrightarrow B^s_{p,
  \max(p,q)}(\Rn). 
\end{equation}
\end{remark}
%}

It is well-known that embeddings of type
\[
\id_{\rn} : \Ae(\rn) \hookrightarrow \Az(\rn)
\]
can never be compact, so in the sequel we turn our attention to embeddings of function spaces on domains which admit compact -- and even nuclear -- embeddings.

\subsubsection*{Function spaces on  domains}

%\bigskip
%BEGIN NEW

%\subsection{Function spaces on arbitrary domains}
Let $\Omega$ be an open set in $\R^d$ such that $\Omega\not= \R^d$. Such a set will  be called an arbitrary domain. {We  denote the collection of  all distributions on $\Omega$ by $\cD'(\Omega)$. }
%We are interested in function spaces on such domains $\Omega$ which have also a wavelet characterization similar to that in Theorem \ref{waveweight}. 
As usual the spaces  $B^s_{p,q}(\Omega)$ and $F^s_{p,q}(\Omega)$ are defined on $\Omega$ by restriction, i.e., using our above convention,
\[
\A(\Omega) = \{f\in \cD'(\Omega): \ \exists\ g\in \A(\rn)\quad \text{with}\ g|_\Omega = f\},
\]
equipped with the quotient norm, as usual,
\[
\| f| \A(\Omega)\| = \inf\left\{ \|g|\A(\rn)\|: \ g\in\A(\rn)\quad \text{with}\ g|_\Omega = f\right\}.
\]
Consequently we have counterparts of the embeddings mentioned in Remark~\ref{A-emb}.

%and   $A^s_{p,q}(\R^n)$, $A^s_{p,q}(\Omega)$ with $A=B$ or $A=F$.

In the classical setting of bounded Lipschitz domains the compactness of Sobolev embeddings is well known and  the following proposition holds, {cf. \cite[Proposition 4.6]{T-F3}.}
%We assume that the reader is familiar with definitions and basic  facts concerning Besov spaces $B^s_{p,q}(\R^n)$ and Triebel-Lizorkin spaces  $F^s_{p,q}(\R^n)$ defined on $\R^n$ as well as the Besov and Triebel-Lizorkin spaces, $B^s_{p,q}(\Omega)$ and $F^s_{p,q}(\Omega)$, defined on $\Omega$ by restrictions. All we need can be found in the first chapter of \cite{T08}. We will use the common notations. In particular we put and 
\begin{proposition}\label{prop-spaces-dom}
  Let $\Omega\subset\rn$ be a bounded Lipschitz domain and 
$s_i\in\real$, $0<p_i,q_i\leq\infty$ ($p_i<\infty$ if $A$=$F$), $i=1,2$. Then 
\begin{equation}\label{id_Omega}
  \id_\Omega : \Ae(\Omega) \to \Az(\Omega)
\end{equation}
is compact, if, and only if,
\begin{equation}\label{id_Omega-comp}
s_1-s_2 > \nd\left(\frac{1}{p_1}-\frac{1}{p_2}\right)_+\ .
\end{equation}
\end{proposition}

In this paper we shall essentially work with a more  general class of domains, so-called {\em quasi-bounded domains},  that still guarantee the compactness of Sobolev embeddings of the above type. We consider this subject in detail in Section~\ref{sect-fs} below.

\begin{remark}
Note that for $\beta_j= 2^{js}$ and $M_j\sim 2^{jd}$, with $d\in \N$, the above sequence space $\ell_q(\beta_j \ell_p^{M_j})= \ell_q(2^{js}\ell_p^{2^{jd}})$ is  isomorphic to  the subspace $\bar{B}^s_{p,q} (\Omega)$ of the Besov space $B^s_{p,q}(\Omega)$ defined on a bounded Lipschitz domain $\Omega$ in $\R^d$. One can prove it using the wavelet decomposition, {cf. \cite[Sections~4.2.4, 4.2.5]{T08}}. In Section~\ref{sect-fs} we study the more general setting of quasi-bounded domains which require the more general approach of sequence spaces as introduced above.
\end{remark}

%%%%%%%%%%%%%%%%%%%%%%%%5

%\bigskip
%BEGIN NEW

%\bigskip

%%%%%%%%%%%%%%%%%%%5

\section{Nuclear embeddings in general vector-valued sequence spaces}\label{sect-nuc}
Our first main goal in this paper is to study nuclear embeddings between sequence spaces of the type $\ell_q(\beta_j \ell_p^{M_j})$ introduced above. So we first recall some fundamentals of the concept and important results we rely on in the sequel.

\subsection{Basic facts concerning  nuclearity }\label{subsec-conc-nuc}
%{The concept of nuclearity and some recent results}\label{subsec-conc-nuc}
%Let $X,Y$ be Banach spaces, $T\in \mathcal{L}(X,Y)$ a linear and bounded operator. %Then $T$ is called {\em nuclear}, denoted by $T\in\mathcal{N}(X,Y)$, if there exist %elements $a_j\in X'$, the dual space of $X$, and $y_j\in Y$, $j\in\mathbb{N}$, such %that $\sum_{j=1}^\infty \|a_j\|_{X'} \|y_j\|_Y < \infty$ and a nuclear %representation $Tx=\sum_{j=1}^\infty a_j(x) y_j$ for any $x\in X$. Together with %the {\em nuclear norm}
%\[
% \nn{T}=\inf\Big\{ \sum_{j=1}^\infty   \|a_j\|_{X'} \|y_j\|_Y:\ T %=\sum_{j=1}^\infty a_j(\cdot) y_j\Big\},
%  \]
%  where the infimum is taken over all nuclear representations of $T$, the space %$\mathcal{N}(X,Y)$ becomes a Banach space. It is obvious that any nuclear operator %can be approximated by finite rank operators, hence 
%  nuclear operators are, in particular, compact.
{Let $X,Y$ be Banach spaces, $T\in \mathcal{L}(X,Y)$ a linear and bounded operator. We have already recalled  the notion of nuclearity in the Introduction. 
%\begin{remark} 
This concept has been introduced by Grothendieck \cite{grothendieck} and was intensively studied afterwards, cf. \cite{Pie-snumb,Pie-op-2,pie-84} and also \cite{pie-history} for some history. There exist extensions of the concept to $r$-nuclear operators, $0<r<\infty$, where $r=1$ refers to the nuclearity. It is well-known that %$\mathcal{N}(X,Y)$ 
the class  of nuclear operators  possesses the ideal property. In Hilbert spaces $H_1,H_2$, the nuclear operators $\mathcal{N}(H_1,H_2)$ coincide with the trace class $S_1(H_1,H_2)$, consisting of those $T$ with singular numbers $(s_n(T))_n \in \ell_1$.
%|\end{remark}

In the next proposition we recall well-known properties of nuclear operators needed in the sequel, cf. \cite[Chapter 1]{jameson}. }

\begin{proposition}\label{coll-nuc}
\benu[\upshape\bfseries (i)]
\item  If $X$ is an $n$-dimensional Banach space, $n\in\N$, then $\ \nn{\id:X\rightarrow X}= n$.  
\item  For any Banach space $X$ and any bounded linear operator $T:\ell^n_\infty\rightarrow X$ we have 
\[\nn{T} = \sum_{i=1}^n \|Te_i\| .\]
\item  If $T\in \mathcal{L}(X,Y)$ is a nuclear operator and $S\in \mathcal{L}(X_0,X)$ and $R\in \mathcal{L}(Y,Y_0)$, then $RTS$ is a nuclear operator and 
\[ \nn{RTS} \le \|R\| \nn{T} \|S\| \qquad \text{(ideal property)} . \] 
\eenu
\end{proposition}
  
Already in the early years there was a strong interest to find {natural} examples of nuclear operators beyond diagonal operators in $\ell_p$ spaces, for which a complete answer was obtained in \cite{tong}. Let $\tau=(\tau_j)_{j\in\nat}$ be a scalar sequence and denote by $D_\tau$ the corresponding diagonal operator, $D_\tau: x=(x_j)_j \mapsto (\tau_j x_j)_j$, acting between $\ell_p$ spaces. Let us introduce the following notation: for  $r_1,r_2\in [1,\infty]$, let $\tn(r_1,r_2)$ be given by 
\begin{equation}\label{tongnumber}
\frac{1}{\tn(r_1,r_2)} = \begin{cases}
    1, & \text{if}\ 1\leq r_2\leq r_1\leq \infty, \\
    1-\frac{1}{r_1}+\frac{1}{r_2}, & \text{if}\ 1\leq r_1\leq r_2\leq \infty.
  \end{cases}
\end{equation}
Hence $1\leq \tn(r_1,r_2)\leq \infty$, and 
\[ \frac{1}{\tn(r_1,r_2)}= 1-\left(\frac{1}{r_1}-\frac{1}{r_2}\right)_+ \geq \frac{1}{r^\ast}= \left(\frac{1}{r_2}-\frac{1}{r_1}\right)_+\ ,\]
with $\tn(r_1,r_2)=r^\ast $ if, and only if, $\{r_1,r_2\}=\{1,\infty\}$.

We heavily rely in our arguments below on the following remarkable result by Tong \cite{tong}.

\begin{proposition}[{\cite[Thms.~4.3, 4.4]{tong}}]\label{prop-tong}
  Let $1\leq r_1,r_2\leq\infty$ and $D_\tau$ be the above diagonal operator.
\benu[\bfseries\upshape (i)]
\item
  Then $D_\tau$ is nuclear if, and only if, $\tau=(\tau_j)_j \in \ell_{\tn(r_1,r_2)}$, with $\ell_{\tn(r_1,r_2)}= c_0$ if $\tn(r_1,r_2)=\infty$. Moreover,
  \[
  \nn{D_\tau:\ell_{r_1}\to\ell_{r_2}} = \|\tau|{\ell_{\tn(r_1,r_2)}}\|.
  \]
\item
  Let $n\in\nat$ and $D^n_\tau: \ell^n_{r_1}\to \ell^n_{r_2}$ be the corresponding diagonal operator $D_\tau^n: x=(x_j)_{j=1}^n \mapsto (\tau_j x_j)_{j=1}^n$. Then 
\begin{equation}\label{tong-res}
\nn{D_\tau^n:\ell_{r_1}^n\rightarrow \ell^n_{r_2}} = \left\| (\tau_j)_{j=1}^n | {\ell_{\tn(r_1,r_2)}^n} \right\|.
\end{equation}
\eenu
\end{proposition}

\begin{example}
In the special case of $\tau\equiv 1$, i.e., $D_\tau=\id$, (i) is not applicable and (ii) reads as  
  \[
\nn{\id :\ell_{r_1}^n\rightarrow \ell^n_{r_2}} =
\begin{cases} 
n & \text{if}\qquad 1\le r_2\le r_1\le \infty,\\
n^{1-\frac{1}{r_1}+\frac{1}{r_2}} & \text{if}\qquad 1\le r_1\le r_2\le \infty .
\end{cases}
\]
In particular, $\nn{\id:\ell_1^n\rightarrow \ell^n_\infty}=1$. 
\end{example}

\begin{remark}
We refer also to \cite{Pie-op-2} for the case $r_1=1$, $r_2=\infty$. 
  \end{remark}

\subsection{Nuclearity results for general vector-valued sequence spaces}

Our aim now is to prove some `nuclear' counterpart of the compactness result recalled in Remark~\ref{rem-compact-seq}.  
%, Proposition~\ref{compact-seq}. 
Roughly speaking, this will read as follows. Assume that $1\leq p_i, q_i\leq\infty$, $i=1,2$. Then $\id_\beta$ given by \eqref{id_beta} is nuclear if, and only if, the compactness condition \eqref{cond-comp-id_beta} is replaced by

%\begin{proposition}\label{nucl-seq-sp}
%  Let $1\leq p_1,p_2\leq\infty$, $1\leq q_1,q_2\leq\infty$, $\{\beta_j\}_{j\in\no}$ be an arbitrary weight sequence and $\{M_j\}_{j\in\no}$ be a sequence of natural numbers. Then the embedding
%  \begin{equation}
%\id_\beta : \ell_{q_1}\left(\beta_j \ell_{p_1}^{M_j}\right) \rightarrow 
%\ell_{q_2}\left(\ell_{p_2}^{M_j}\right)
%  \end{equation}
%  is nuclear if, and only if,
  \begin{align}
    \left(\beta_j^{-1} M_j^{\frac{1}{\tn(p_1,p_2)}}\right)_{j\in\no} \in \ell_{\tn(q_1,q_2)},
        \end{align}
where for $\tn(q_1,q_2)=\infty$ the space $\ell_\infty$ has to be replaced by $c_0$.
%\end{proposition}

But we need some preparation to prove (an even more general version of) this result and postpone it as Theorem~\ref{nucl-seq-sp} below, together with some further discussion. In our argument below we rely on the approach in Tong's paper \cite{tong} and adapt and extend it appropriately. 

%    \subsection{ Idea of the proof of Proposition \ref{nucl-seq-sp}}
    
\smallskip~

   We consider the finite-dimensional version of the spaces introduced in Definition   \ref{seq-sp-def1}. For $n\in \No$, {$1\leq p,q\leq\infty$}, we put   
    \begin{align} \label{seq-sp-def2}
    \ell_q^{ n} \left(\beta_j \ell_p^{M_j}\right) =  \Big\{ & x=(x_{j,k})_{0\le j\le n, k=1, \dots, M_j}:  \ {x_{j,k}\in\C},  \\ \nonumber
& \left\|x|\ell_q\left(\beta_j \ell_p^{M_j}\right)\right\| = \Big(\sum_{j=0}^n\beta_j^q \big(\sum_{k=1}^{M_j} |x_{j,k}|^p\Big)^{\frac{q}{p}}\Big)^{\frac1q}
    \Big\},
  \end{align}
{appropriately modified for $p=\infty$ and/or $q=\infty$}. 
    Following the ideas of the proof in  \cite{tong}, we  are interested in embeddings of these spaces or, equivalently, in actions of  diagonal operators on the spaces with $\beta_j\equiv 1$.  More generally we will work with  operators acting on  spaces given by  matrices. Therefore it will be convenient  to rewrite the above definition  in the following way. Let $M=(M_j)_{j\in\No}$  be a sequence of natural numbers and let $n\in \N_0$. Denote by {$\alpha_j=\sum_{l=0}^{j-1} M_l$, $j\in\no$, with $M_{-1}:=0$, i.e., $\alpha_0=0$. Let  
   $I_j=\{\alpha_j+1, \ldots , \alpha_{j+1}\}$, that is, $I_0=\{1, \dots, M_0\}$, and  $N=\alpha_{n+1}$}. We put 
%\[ {\frak l}_q({\frak l}^{M_j}_p)\qquad {\frak l}^n_q({\frak l}^{M_j}_p)\]
\begin{align} 
\label{seq-sp-def2a}
    % X_{p,q}^{M}  
{    {\mathfrak l}_q({\mathfrak l}^{M_j}_p)}= & \left\{x=(x_{k})_{k\in  \N} :  {x_k\in\C},
 \left\|x|{\mathfrak l}_q({\mathfrak l}^{M_j}_p)\right\| = \Big(\sum_{j=0}^\infty \big(\sum_{k\in I_j} |x_{k}|^p\Big)^{\frac{q}{p}}\Big)^{\frac1q}<\infty
    \right\},\\
    \label{seq-sp-def2b}
    %X_{p,q}^{n, M}  
{  {\mathfrak l}^n_q({\mathfrak l}^{M_j}_p) }= & \left\{x=(x_{k})_{1\le k\le N} :  {x_k\in\C},
 \left\|x| {\mathfrak l}^n_q({\mathfrak l}^{M_j}_p)\right\| = \Big(\sum_{j=0}^n \big(\sum_{k\in I_j} |x_{k}|^p\Big)^{\frac{q}{p}}\Big)^{\frac1q}
    \right\}, 
\end{align}
{appropriately modified for $p=\infty$ and/or $q=\infty$}. 
 Then $\ds {\mathfrak l}_q({\mathfrak l}^{M_j}_p)$ is a vector space isometrically isomorphic to $\ell_{q}( \ell_{p}^{M_j})$ and $\ds {\mathfrak l}^n_q({\mathfrak l}^{M_j}_p)$ is a finite-dimensional vector space isometrically isomorphic to $\ds\ell_{q}^n( \ell_{p}^{M_j})$, with $\ds\dim  {\mathfrak l}^n_q({\mathfrak l}^{M_j}_p)= N$. 
 
 \begin{lemma}\label{oper-norm}
Let $n\in\No$, $1\leq p_i,q_i\leq\infty$, $i=1,2$, and $M=(M_j)_{j\in\No}$ and $N$ as above. Let $\lambda = (\lambda_k)_{k=1, \dots, N}$ be complex numbers and 
%$D_\lambda: X_{p_1,q_1}^{n, M} \rightarrow  X_{p_2,q_2}^{n, M}$
$D_\lambda: { \ds {\mathfrak l}^n_{q_1}({\mathfrak l}^{M_j}_{p_1}) }
%X_{p_1,q_1}^{n, M}
 \rightarrow  { \ds {\mathfrak l}^n_{q_2}({\mathfrak l}^{M_j}_{p_2}) }
 %X_{p_2,q_2}^{n, M}
 $ be the diagonal operator given by $D_\lambda:(x_k)_{k=1}^N\mapsto (\lambda_k x_k)_{k=1}^N$. Then 
\begin{align*} 
\| D_\lambda: {\mathfrak l}^n_{q_1}({\mathfrak l}^{M_j}_{p_1}) \to {\mathfrak l}^n_{q_2}({\mathfrak l}^{M_j}_{p_2})\| = \|\lambda| {  {\mathfrak l}^n_{q^*}({\mathfrak l}^{M_j}_{p^*}) }\|, %X_{p^*,q^*}^{n, M} \|,
\quad \lambda = (\lambda_k)_{k=1}^N,
\end{align*}
 where $p^*$ and $q^*$ are given by \eqref{pq-star}.
 \end{lemma} 

%%%%%%%%%%%%%%%%%%%%%%%%%%%%%%%%%%%%%%%%%%%%%%%%%%%%%%%%%%%%%%%%%%%%%
 \begin{proof}
   The upper estimate of $\| D_\lambda\| = \| D_\lambda: \ds {\mathfrak l}^n_{q_1}({\mathfrak l}^{M_j}_{p_1}) \to {\mathfrak l}^n_{q_2}({\mathfrak l}^{M_j}_{p_2})\|$, that is, 
     $\| D_\lambda\| \le \|\lambda| {\ds{\mathfrak l}^n_{q^*}({\mathfrak l}^{M_j}_{p^*}) } \|$, follows easily from H\"older's inequality if $p_2<p_1$ or $q_2<q_1$, and from the monotonicity of the $\ell_r$ norms otherwise. To prove the opposite inequality one can find  sequences that correspond to the equalities in the H\"older inequalities. For details we refer to  \cite{BZB} where a similar, even more general statement is proved.
\end{proof}
%%%%%%%%%%%%%%%%%%%5

% \begin{remark}
%A more general statement, similar to the above lemma, can be found in \cite{BZB}. %We have given here the above argument for completeness. 
%\end{remark}
%%%%%%%%%%%%%%%%%%%%%%%%%%%%%%%%%%%%%%%%%%%%%%%%%%%%%%%%%%%%%%%%%%%%%%%%%

Let $A$ be an $N\times N$ complex  matrix. Let $\mathcal{D}(A)$ denote the diagonal part of $A$, i.e., a matrix  that derives  from $A$ by  replacing  all off-diagonal entries of $A$ by zeros. Analogously, if $T$ is a linear operator on  ${\mathbb C}^N$  given by the matrix $A$, then $\mathcal{D}(T)$ will denote the operator given by the matrix $\mathcal{D}(A)$. 

\begin{lemma}\label{diagnorm}
Let $n\in\No$, $1\leq p_i,q_i\leq\infty$, $i=1,2$, and $M=(M_j)_{j\in\No}$ as above. If %$T: X_{p_1,q_1}^{n, M} \rightarrow X_{p_2,q_2}^{n, M} $ 
$\ds T: {\mathfrak l}^n_{q_1}({\mathfrak l}^{M_j}_{p_1}) \rightarrow    {\mathfrak l}^n_{q_2}({\mathfrak l}^{M_j}_{p_2}) $ 
is a linear operator, then 
\[ \| \mathcal{D}(T)\|\le \|T\| .\] 
\end{lemma}

\begin{proof}
{ Let $A$ be the matrix of $T$. 
Let $\omega\in \C$ 
%$\omega= e^{2\pi\sqrt{-1}/N}$ 
and let $U(\omega)$ be the diagonal matrix with entries $1, \omega, \omega^2, \dots ,\omega^{N-1}$ down its diagonal. If $|\omega|=1$, then the matrix $U(\omega)$ generates an operator in $\C^N$ that is an isometry in the norm of   $ \ds{\mathfrak l}^n_{q^*}({\mathfrak l}^{M_j}_{p^*}) $ for all $p,q$. We take $\omega= e^{2\pi i/N}$.     
Using the identity $\sum_{j=0}^{N-1}\omega^j=0$ and elementary calculations one concludes, cf. \cite{Bh}, that
\begin{equation}
\mathcal{D}(A) = \frac{1}{N} \sum_{j=0}^{N-1} U(\omega)^j A U(\bar{\omega})^j . 
\end{equation}
%where $U^*$ denotes the adjoint matrix to $U$, .  
The operators $U(\omega)^j$ and $U(\bar{\omega})^j$ are isometries therefore the last formula implies $\| \mathcal{D}(T)\|=\| \mathcal{D}(A)\|\le \|A\|=\|T\|$.  } 
%Let $v \in  X_{p_1,q_1}^{n, M} $ be a vector of norm $1$.  Then $U^{*j}v$ is also a %vector of norm $1$ in  $ X_{p_1,q_1}^{n, M}$ and $U^jAU^{*j}v$ is a vector in  $ %X_{p_2,q_2}^{n, M} $ with the same norm as  $AU^{*j} v$. So 
%\begin{align*}
%  \| \mathcal{D}(T)\| & = \sup\left\{\|\mathcal{D}(A) v| X_{p_2,q_2}^{n, M}\| :\; %\|v|X_{p_1,q_1}^{n, M} \|=1\right\} \\
%  & \le  
%  \frac{1}{N} \sum_{j=0}^{N-1}    \sup \left\{\|U^jAU^{*j}v|X_{p_2,q_2}^{n, M} \| %:\;    \|v|X_{p_1,q_1}^{n, M} \|=1 \right\}  \\
%  &\le  \|A\| = \|T\|.
%\end{align*}
\end{proof}

%%%%%%%%%%%%%%%%%%%%%%%%%%%%%%%%%%%%%%%%%%%%%%%%%%
%To calculate the nuclear norm of the diagonal operators we will use the relation %between the nuclear operators and projective tensor products of Banach spaces, e.g. %cf. \cite[Chapter 2]{ryan}. We recall that the projective tensor product 
%$X \hat{\otimes}_\pi Y $ of two Banach spaces $X$ and $Y$ is a completion  of the %tensor product $X\otimes Y$ in the norm 
%\[
%\pi(u)= \inf \left\{ \sum_{k=1}^n \| x_k\|_X \|y_k\|_Y:\; u= \sum_{k=1}^n %x_k\otimes y_k  \right\} . 
%\]
%If $X$ and $Y$ are finite-dimensional, then  the space $X\otimes Y$ equipped with %the above norm is complete. 
%
%Let    $\mathcal{L}(X,Y)$ denote  the space of bounded linear operators from  $X$ %to  $Y$. If either $X'$ or $Y$ has the approximation property, then there is a %one-to-one linear map of  $X'\otimes Y$  into the space  $\mathcal{L}(X,Y)$ and it %induces an isometric isomorphism of   $X'\otimes Y$ onto $T\in\mathcal{N}(X,Y)$.   %We will need also the following result that is due to Grothendieck %\cite{grothendieck}, cf. also \cite{tong}.
%\begin{lemma}\label{gr}
%The map $L: (X'\hat{\otimes}_\pi Y)'\rightarrow \mathcal{L}(X',Y')$ defined, for %each $ Q\in (X'\hat{\otimes}_\pi Y)'$ by setting $L_Q$ to be the operator %$(L_Q(x'), y) = Q(x'\otimes y)$ for all $x'\otimes y\in X'\hat{\otimes}_\pi Y$ is %an isometric isomorphism. 
%\end{lemma}
%%%%%%%%%%%%%%%%%%%%%%%%%%%%%%%%%%%%%%%%%%%%%%%%%%%%%%%%%%

\begin{lemma}\label{nucD} 
Let $1\le p_i, q_i\le \infty$, $i=1,2$, $n\in\No$ and $M=(M_j)_{j\in\No}$ as above. Assume that %$D:X_{p_1,q_1}^{n, M}\rightarrow X_{p_2,q_2}^{n, M} $ 
$D:  \ds {\mathfrak l}^n_{q_1}({\mathfrak l}^{M_j}_{p_1}) \rightarrow   {\mathfrak l}^n_{q_2}({\mathfrak l}^{M_j}_{p_2}) $
is a diagonal linear operator defined by the diagonal matrix $A$ with entries $\lambda_{1},\dots, \lambda_{N}$. Then
\[ 
\nn{D}= \|\lambda |\mathfrak{l}^n_{\tn(q_1,q_2)}(\mathfrak{l}^{M_j}_{\tn(p_1,p_2)}) \| ,
\quad \lambda = (\lambda_k)_{k=1}^N,
\]
where $\tn(p_1,p_2)$ and $\tn(q_1,q_2)$ are given by \eqref{tongnumber}.
\end{lemma}
\begin{proof}
 \emph{Step 1}.\quad
Both spaces $\ds{\mathfrak l}^n_{q_1}({\mathfrak l}^{M_j}_{p_1})$ and  $\ds  {\mathfrak l}^n_{q_2}({\mathfrak l}^{M_j}_{p_2}) $ are finite-dimensional, therefore it is sufficient to consider the finite representations of any nuclear  operator $T$,  i.e., 
\begin{equation}\label{nucnrm1}
T = \sum_{\ell=1}^k x'_\ell\otimes y_\ell, \qquad x'_\ell\in ( {\mathfrak l}^n_{q_1}({\mathfrak l}^{M_j}_{p_1}))' \quad\text{and}\quad y_\ell\in  {\mathfrak l}^n_{q_2}({\mathfrak l}^{M_j}_{p_2}), 
\end{equation}
cf. \cite[p.19]{jameson}.  

First we show that  
the  space $\ds \mathcal{N}\big( {\mathfrak l}^n_{q_1}({\mathfrak l}^{M_j}_{p_1}), {\mathfrak l}^n_{q_2}({\mathfrak l}^{M_j}_{p_2})\big)'$, that is, the dual space to $\ds \mathcal{N}\big( {\mathfrak l}^n_{q_1}({\mathfrak l}^{M_j}_{p_1}), {\mathfrak l}^n_{q_2}({\mathfrak l}^{M_j}_{p_2})\big)$, is isometrically isomorphic to $\ds\mathcal{L} \big(
{\mathfrak l}^n_{q_2}({\mathfrak l}^{M_j}_{p_2}), {\mathfrak l}^n_{q_1}({\mathfrak l}^{M_j}_{p_1})\big)$. 
This can be proved via  trace-duality.
For any operator $S\in \mathcal{L} \ds\big(
{\mathfrak l}^n_{q_2}({\mathfrak l}^{M_j}_{p_2}), {\mathfrak l}^n_{q_1}({\mathfrak l}^{M_j}_{p_1})\big)$ the composition operator $ST$ is again a nuclear operator which belongs to  $\ds\mathcal{N}\big( {\mathfrak l}^n_{q_1}({\mathfrak l}^{M_j}_{p_1}), {\mathfrak l}^n_{q_1}({\mathfrak l}^{M_j}_{p_1})\big)$, and
\begin{equation} \nonumber
ST = \sum_{\ell=1}^k x'_\ell\otimes Sy_\ell, \qquad x'_\ell\in ( {\mathfrak l}^n_{q_1}({\mathfrak l}^{M_j}_{p_1}))' \quad\text{and}\quad Sy_\ell\in  {\mathfrak l}^n_{q_1}({\mathfrak l}^{M_j}_{p_1}) .
\end{equation}
Then
\[ 
\varphi_S(T) := \tr (ST) =   \sum_{\ell=1}^k x'_\ell (Sy_\ell)
\] 
defines a linear functional on $\ds\mathcal{N}\big( {\mathfrak l}^n_{q_1}({\mathfrak l}^{M_j}_{p_1}), {\mathfrak l}^n_{q_2}({\mathfrak l}^{M_j}_{p_2})\big)$. This is well-defined as the trace does not depend on the particular  representation of $ST$ and $T$, respectively.
Moreover, due to the ideal property of nuclear operators, recall Proposition~\ref{coll-nuc}(iii),  we have 
\[ \|\varphi_S|\mathcal{N}\big( {\mathfrak l}^n_{q_1}({\mathfrak l}^{M_j}_{p_1}), {\mathfrak l}^n_{q_2}({\mathfrak l}^{M_j}_{p_2})\big)'\|\le \|S| \mathcal{L} \big(
{\mathfrak l}^n_{q_2}({\mathfrak l}^{M_j}_{p_2}), {\mathfrak l}^n_{q_1}({\mathfrak l}^{M_j}_{p_1})\big)\|.\] 
The reverse inequality can be proved by applying the functional $\varphi_S$ to rank-one operators $T=x'\otimes y$ with $\|x' \|=\|y\|=1$. 
 Thus the mapping
 \[
\Phi : \mathcal{L}\big(
 {\mathfrak l}^n_{q_2}({\mathfrak l}^{M_j}_{p_2}), {\mathfrak l}^n_{q_1}({\mathfrak l}^{M_j}_{p_1})\big) \rightarrow   \mathcal{N}\big( {\mathfrak l}^n_{q_1}({\mathfrak l}^{M_j}_{p_1}), {\mathfrak l}^n_{q_2}({\mathfrak l}^{M_j}_{p_2})\big)', \quad 
 S \mapsto \varphi_S ,\]      
is an isometry. 
 It should be clear that both spaces have the same dimension and so the mapping is an isometric isomorphism. This implies that
 \[ \|\varphi_S\| = \|\varphi_S|\mathcal{N}\big( {\mathfrak l}^n_{q_1}({\mathfrak l}^{M_j}_{p_1}), {\mathfrak l}^n_{q_2}({\mathfrak l}^{M_j}_{p_2})\big)'\| = \|S| \mathcal{L} \big(
    {\mathfrak l}^n_{q_2}({\mathfrak l}^{M_j}_{p_2}), {\mathfrak l}^n_{q_1}({\mathfrak l}^{M_j}_{p_1})\big)\| = \|S\|,
    \]
such that, in particular,  
 \begin{align} \nonumber
	\nu(D) = & \sup\{|\varphi(D)|: \varphi \in \mathcal{N}\big( {\mathfrak l}^n_{q_1}({\mathfrak l}^{M_j}_{p_1}), {\mathfrak l}^n_{q_2}({\mathfrak l}^{M_j}_{p_2})\big)', \quad \|\varphi\|\le 1\} \\ \nonumber
	 = & \sup\{|\tr (SD)|:  S \in {\mathcal L}\big( {\mathfrak l}^n_{q_2}({\mathfrak l}^{M_j}_{p_2}), {\mathfrak l}^n_{q_1}({\mathfrak l}^{M_j}_{p_1})\big), \quad \|S\|\le 1\} . %\\ \label{nucnrm2}
	% = & \sup\{|(D, D(S))|:  S \in {\mathcal L}\big( {\mathfrak l}^n_{q_2}({\mathfrak l}^{M_j}_{p_2}), {\mathfrak l}^n_{q_1}({\mathfrak l}^{M_j}_{p_1})\big), \quad \|S\|\le 1\} ~.
\end{align}

\emph{Step 2.}\quad% We have $( {\mathfrak l}^n_{q_1}({\mathfrak l}^{M_j}_{p_1}))'=  {\mathfrak l}^n_{q'_1}({\mathfrak l}^{M_j}_{p'_1})$ and $( {\mathfrak l}^n_{q_2}({\mathfrak l}^{M_j}_{p_2}))'=  {\mathfrak l}^n_{q'_2}({\mathfrak l}^{M_j}_{p'_2})$.
Let $S \in \ds{\mathcal L}\big( {\mathfrak l}^n_{q_2}({\mathfrak l}^{M_j}_{p_2}), {\mathfrak l}^n_{q_1}({\mathfrak l}^{M_j}_{p_1})\big)$ with $\|S\|\leq 1$. We denote its diagonal part by $\mathcal{D}(S) : \ds {\mathfrak l}^n_{q_2}({\mathfrak l}^{M_j}_{p_2}) \rightarrow {\mathfrak l}^n_{q_1}({\mathfrak l}^{M_j}_{p_1})$, represented by the entries $b=(b_\ell)_{\ell=1}^N$. 
 Lemma \ref{oper-norm} shows that the subspace of diagonal operators in 
 $\ds\mathcal{L}\big( {\mathfrak l}^n_{q_2}({\mathfrak l}^{M_j}_{p_2}),{\mathfrak l}^n_{q_1}({\mathfrak l}^{M_j}_{p_1})\big)$ is isometrically isomorphic to  $\ds {\mathfrak l}_{\tilde{q}^\ast}^n({\mathfrak l}^{M_j}_{\tilde{p}^\ast})$, and
 \[ \|\mathcal{D}(S) | \mathcal{L}\big( {\mathfrak l}^n_{q_2}({\mathfrak l}^{M_j}_{p_2}),{\mathfrak l}^n_{q_1}({\mathfrak l}^{M_j}_{p_1})\big)\| = \| b | {\mathfrak l}_{\tilde{q}^\ast}^n({\mathfrak l}^{M_j}_{\tilde{p}^\ast})\|,
 \]
 where
 \[
 \frac{1}{\tilde{p}^\ast}:=\left(\frac{1}{p_1}-\frac{1}{p_2}\right)_+ \quad\text{and}\quad 
 \frac{1}{\tilde{q}^\ast}:=\left(\frac{1}{q_1}-\frac{1}{q_2}\right)_+ .
 \]
 Note that 
 \[
 \frac{1}{\tilde{p}^\ast} = 1-\frac{1}{\tn(p_1,p_2)} = \frac{1}{\tn(p_1,p_2)'}\quad\text{and}\quad 
 \frac{1}{\tilde{q}^\ast} = 1-\frac{1}{\tn(q_1,q_2)} = \frac{1}{\tn(q_1,q_2)'}.
 \]
 Consequently, 
\begin{align}\label{100621-2}
  \|\mathcal{D}(S) | \mathcal{L}\big( {\mathfrak l}^n_{q_2}({\mathfrak l}^{M_j}_{p_2}),{\mathfrak l}^n_{q_1}({\mathfrak l}^{M_j}_{p_1})\big)\| =
  \| b | {\mathfrak l}^n_{\tn(q_1,q_2)'}({\mathfrak l}^{M_j}_{\tn(p_1,p_2)'})\| .
\end{align}
Since $\tr(SD) = \tr(\mathcal{D}(S)D)$ and $\|\mathcal{D}(S)\|\leq \|S\|$ by Lemma~\ref{diagnorm}, we get for any such operator $S$ with $\|S\|\leq 1$ by H\"older's inequality that
\begin{align*}
  |\tr(SD)| & = |\tr(\mathcal{D}(S)D)| = |\sum_{\ell=1}^N {\lambda_{ \ell}}b_{ \ell}| \\
  & \leq \|\lambda |  {\mathfrak l}^n_{\tn(q_1,q_2)}({\mathfrak l}^{M_j}_{\tn(p_1,p_2)})\|\  \| b | {\mathfrak l}^n_{\tn(q_1,q_2)'}({\mathfrak l}^{M_j}_{\tn(p_1,p_2)'})\| \\
  & = \|\lambda |  {\mathfrak l}^n_{\tn(q_1,q_2)}({\mathfrak l}^{M_j}_{\tn(p_1,p_2)})\|\  \|\mathcal{D}(S)\| \\
  & \leq \|\lambda |  {\mathfrak l}^n_{\tn(q_1,q_2)}({\mathfrak l}^{M_j}_{\tn(p_1,p_2)})\|. 
\end{align*}
Thus the first step implies that
  \begin{align}\nonumber
    \nn{D} =  & \ \sup \{|\tr(S D)| : \
     %|\sum_{\ell=1}^N {\lambda_{ \ell}}b_{ \ell}| : \;  \| (b_{\ell})_{\ell=1}^N | {\mathfrak l}^n_{\tn(q_1,q_2)'}({\mathfrak l}^{M_j}_{\tn(p_1,p_2)'})\|
\|S | \mathcal{L}( {\mathfrak l}^n_{q_2}({\mathfrak l}^{M_j}_{p_2})),{\mathfrak l}^n_{q_1}({\mathfrak l}^{M_j}_{p_1})\|
    \le 1\;\} \\
%&  \text{and} \;L_Q\;\text{is a diagonal operator with entries}\; b_{11},\ldots, b_{NN}\}. \nonumber  
\leq & \ \|\lambda |  {\mathfrak l}^n_{\tn(q_1,q_2)}({\mathfrak l}^{M_j}_{\tn(p_1,p_2)})\|.
    \label{100621-1}
  \end{align}
  Conversely, we may restrict ourselves to diagonal operators $S=\mathcal{D}(S)$ in $\ds\mathcal{L}\big( {\mathfrak l}^n_{q_2}({\mathfrak l}^{M_j}_{p_2}),{\mathfrak l}^n_{q_1}({\mathfrak l}^{M_j}_{p_1})\big)$ and benefit from the sharpness of the H\"older inequality to obtain the estimate from below,
  $$\nn{D} \geq \|\lambda |  {\mathfrak l}^n_{\tn(q_1,q_2)}({\mathfrak l}^{M_j}_{\tn(p_1,p_2)})\|.$$
  This concludes the argument.
\end{proof}

%\begin{theorem}
\begin{proposition}
\label{main-seq-nuc}
  Let $1\le p_i, q_i\le \infty$, $i=1,2$, {$M=(M_j)_{j\in\No}$ as above},  and let %$D_\lambda:X_{p_1,q_1}^{M} \rightarrow X_{p_2,q_2}^{M}$ 
  $\ds D_\lambda: \mathfrak{l}_{q_1}(\mathfrak{l}^{M_j}_{p_1}) \rightarrow \mathfrak{l}_{q_2}(\mathfrak{l}^{M_j}_{p_2})$ be a diagonal linear operator defined by the sequence  $\lambda=(\lambda_{j})_{j\in\nat}$. %1,\lambda_2, \ldots )$.
  Then the operator $D_\lambda$ is nuclear if, and only if,  
  $\ds \lambda \in \mathfrak{l}_{\tn(q_1,q_2)}(\mathfrak{l}^{M_j}_{\tn(p_1,p_2)})\ $  %$X_{\tn(p_1,p_2),\tn(q_1,q_2)}^{M}\ $ 
  and $\ \lambda \in c_0\ $ 
%$\lambda_j\rightarrow 0$ for  $k\rightarrow \infty$
if  $\tn(q_1,q_2)=\infty$. Moreover,
\[\nn{D_\lambda}= \|\lambda %(\lambda_1,\lambda_2, \ldots )
|\mathfrak{l}_{\tn(q_1,q_2)}(\mathfrak{l}^{M_j}_{\tn(p_1,p_2)})\| .\] 
\end{proposition}
%\end{theorem}

\begin{proof} 
  \emph{Step 1.}\quad First we show that the nuclearity of $D_\lambda$ implies that $\lambda\in\mathfrak{l}_{\tn(q_1,q_2)}(\mathfrak{l}^{M_j}_{p_1,p_2})\ $ with the additional requirement that $\ \lambda \in c_0\ $ if $\tn(q_1,q_2)=\infty$. In particular, we shall obtain that
  \begin{equation}\label{lower_est}
     \nn{D_\lambda} \geq \| \lambda |\mathfrak{l}_{\tn(q_1,q_2)}(\mathfrak{l}^{M_j}_{\tn(p_1,p_2)})\| .
  \end{equation}
Let for $n\in\No$, $\ds P_n: 	\mathfrak{l}_{q_2}(\mathfrak{l}^{M_j}_{p_2})\rightarrow  \mathfrak{l}^n_{q_2}(\mathfrak{l}^{M_j}_{p_2})$ denote a projection given by $(x_1,x_2,\ldots )\mapsto (x_1,\ldots ,x_{N(n)})$ and let  $\ds S_n: \mathfrak{l}^n_{q_1}(\mathfrak{l}^{M_j}_{p_1})\rightarrow \mathfrak{l}_{q_1}(\mathfrak{l}^{M_j}_{p_1})$ denote an embedding given by $(x_1,\ldots ,x_{N(n)})\mapsto (x_1,x_2,\ldots , x_{N(n)}, 0,0,\ldots )$. Both operators have norm $1$.
% We have the following commutative diagram 
%\[
%\begin{CD}
%  X_{p_1,q_1}^{n,M} @>D>> X_{p_2,q_2}^{n,M} \\ 
% @VS_nVV @AAP_nA   \\
% X_{p_1,q_1}^{M} @>D>> X_{p_2,q_2}^{M}  
%\end{CD}
%\]
If %$D_\lambda:X_{p_1,q_1}^{M} \rightarrow X_{p_2,q_2}^{M}$ 
$\ds D_\lambda: \mathfrak{l}_{q_1}(\mathfrak{l}^{M_j}_{p_1}) \rightarrow \mathfrak{l}_{q_2}(\mathfrak{l}^{M_j}_{p_2})$ is a nuclear operator, then by Proposition~\ref{coll-nuc}(iii) 
%$ P_nD_\lambdaS_n:X_{p_1,q_1}^{n,M} \rightarrow X_{p_2,q_2}^{n,M}$ 
$\ds P_nD_\lambda S_n: \mathfrak{l}_{q_1}(\mathfrak{l}^{M_j}_{p_1}) \rightarrow \mathfrak{l}_{q_2}(\mathfrak{l}^{M_j}_{p_2})$ is nuclear and 
\begin{equation}
\nn{P_nD_\lambda S_n} \le  \|P_n\|\nn{D_\lambda}\|S_n\| = \nn{D_\lambda} . 
\end{equation}
Hence Lemma \ref{nucD} implies
\begin{align*}
  \|\lambda %(\lambda_1,\lambda_2, \ldots )
  |\mathfrak{l}_{\tn(q_1,q_2)}(\mathfrak{l}^{M_j}_{\tn(p_1,p_2)})\| & =\lim_{n\rightarrow \infty}\|(\lambda_{1},\lambda_{2}, \ldots, \lambda_{N(n)})| \mathfrak{l}_{\tn(q_1,q_2)}(\mathfrak{l}^{M_j}_{\tn(p_1,p_2)})\| \\
  & = \lim_{n\rightarrow \infty} \nn{ P_nD_\lambda S_n} \le   \nn{D_\lambda}.
  \end{align*}

If $\tn(q_1,q_2)=\infty$, we need to show, in addition, that $\lambda\in c_0$. However, if $\lambda \not\in c_0$, then one can easily prove that the operator $D_\lambda$ is not compact. So it is also not nuclear. \\

\emph{Step 2.}\quad It remains to show the converse direction, i.e., the sufficiency of $\lambda\in\mathfrak{l}_{\tn(q_1,q_2)}(\mathfrak{l}^{M_j}_{\tn(p_1,p_2)})$ for the nuclearity of $D_\lambda$, and the inequality converse to \eqref{lower_est}. 
Suppose that we have a diagonal  operator $D_\lambda$ defined by the   sequence $\lambda\in  \mathfrak{l}_{\tn(q_1,q_2)}(\mathfrak{l}^{M_j}_{\tn(p_1,p_2)})$, with $\lambda\in c_0$ if $\tn(q_1,q_2)=\infty$. Then the sequence $\left((\lambda_{1},\lambda_{2}, \ldots , \lambda_{N(n)},0,0, \dots)\right)_n$ is a Cauchy sequence in $\mathfrak{l}_{\tn(q_1,q_2)}(\mathfrak{l}^{M_j}_{\tn(p_1,p_2)})$. Thus Lemma \ref{nucD} implies that the sequence of  diagonal operators $D^{(n)}_\lambda$ that are defined by the sequences  $\left((\lambda_{1},\lambda_{2}, \ldots , \lambda_{N(n)},0,0, \dots)\right)_n$, is a Cauchy sequence in %$\mathcal{N}( X_{p_1,q_1}^{M}, X_{p_2,q_2}^{M})$. 
$\mathcal{N}(\mathfrak{l}_{q_1}(\mathfrak{l}^{M_j}_{p_1}), \mathfrak{l}_{q_2}(\mathfrak{l}^{M_j}_{p_2})$. The space of nuclear operators is a Banach space, so the sequence $(D^{(n)}_\lambda)_n$ is convergent to a nuclear operator $\widetilde{D}$ in the nuclear norm.  On the other hand one can easily see that
\[ \lim_{n\rightarrow \infty} D^{(n)}_\lambda(x_1,x_2,\ldots )= (\lambda_{1} x_1,\lambda_{2}x_2,\ldots ) = D_\lambda x. \]  
So the sequence $ (D^{(n)}_\lambda)_n$ converges to $\widetilde{D}=D_\lambda$ in the sense of pointwise convergence. The convergence in the sense of the nuclear norm is stronger, so  $ (D^{(n)}_\lambda)_n$ converges to $D_\lambda$ also in $\mathcal{N}\big(
\mathfrak{l}_{q_1}(\mathfrak{l}^{M_j}_{p_1}),  \mathfrak{l}_{q_2}(\mathfrak{l}^{M_j}_{p_2})\big)$ and 
\begin{align*}
  \nn{D_\lambda} & = \lim_{n\rightarrow \infty} \nn{D^{(n)}_\lambda} \\
  &=  \lim_{n\rightarrow \infty} \|(\lambda_{1},\lambda_{2}, \ldots, \lambda_{N(n)},0,0,\dots )| \mathfrak{l}_{\tn(q_1,q_2)}(\mathfrak{l}^{M_j}_{\tn(p_1,p_2)})\| \\
  & =\|\lambda|\mathfrak{l}_{\tn(q_1,q_2)}(\mathfrak{l}^{M_j}_{\tn(p_1,p_2)})\|.
  \end{align*}
\end{proof}

Now we are ready to present our main outcome in this section. It generalises Tong's result \cite{tong} as recalled in Proposition~\ref{prop-tong} to the setting of generalised vector-valued sequence spaces.

\begin{theorem}
\label{nucl-seq-sp}
  Let $1\leq p_i, q_i\leq\infty$, $i=1,2$, $(\beta_j)_{j\in\no}$ be an arbitrary weight sequence and $(M_j)_{j\in\no}$ be a sequence of natural numbers. Then the embedding
  \begin{equation}
\id_\beta : \ell_{q_1}\left(\beta_j \ell_{p_1}^{M_j}\right) \rightarrow 
\ell_{q_2}\left(\ell_{p_2}^{M_j}\right)
  \end{equation}
  is nuclear if, and only if,
  \begin{align}\label{nucseqsp}
     \left(\beta_j^{-1} M_j^{\frac{1}{\tn(p_1,p_2)}}\right)_{j\in\no} \in \ell_{\tn(q_1,q_2)},
        \end{align}
  where for $\tn(q_1,q_2)=\infty$ the space $\ell_\infty$ has to be replaced by $c_0$. In that case,
  \[
  \nn{\id_\beta} { = } \left\|   \left(\beta_j^{-1} M_j^{\frac{1}{\tn(p_1,p_2)}}\right)_{j\in\no} | \ell_{\tn(q_1,q_2)}\right\|.
  \]
%\end{corollary}
\end{theorem}

\begin{proof}
We apply Proposition~\ref{main-seq-nuc}. The space  $\ds \mathfrak{l}_{q}(\mathfrak{l}^{M_j}_{p})$  is isometrically isomorphic to $\ell_q(\ell_p^{M_j})$. So the embedding $\id_\beta$ corresponds to a diagonal operator $D_\lambda$ from $\mathfrak{l}_{q_1}(\mathfrak{l}^{M_j}_{p_1})$ to $\mathfrak{l}_{q_2}(\mathfrak{l}^{M_j}_{p_2})$  with $\lambda_{l} = \beta^{-1}_j$ if $l \in I_j$ and this gives
\[  \nn{\id_\beta} = \nn{D_\lambda} = \|\lambda %(\lambda_1,\lambda_2, \ldots )
|\mathfrak{l}_{\tn(q_1,q_2)}(\mathfrak{l}^{M_j}_{\tn(p_1,p_2)})\|   = \left\|   \left(\beta_j^{-1} M_j^{\frac{1}{\tn(p_1,p_2)}}\right)_{j\in\no} | \ell_{\tn(q_1,q_2)}\right\| . \] 
\end{proof}
%\bigskip~

\begin{remark}
In case of $M_j\equiv 1$, $\beta_j= \tau_j^{-1}$, Theorem~\ref{nucl-seq-sp} generalises Tong's result \cite{tong}, cf. Proposition~\ref{prop-tong}(i), in a natural way.
 \end{remark}
 
\begin{remark}\label{rem-extremal-seq}
Note that in the extremal cases $\{p_1,p_2\} = \{1,\infty\}$ and $\{q_1,q_2\}=\{1,\infty\}$, that is, whenever $\tn(p_1,p_2)=p^\ast$ and $\tn(q_1,q_2)=q^\ast$, then compactness and nuclearity of the embedding $\id_\beta$ coincide,  recall condition \eqref{cond-comp-id_beta}.  %Proposition~\ref{compact-seq}. 
Moreover, when  $p_1=1$ and $p_2=\infty$, then $\tn(p_1,p_2)=\infty$ and  condition \eqref{nucseqsp} is reduced to $(\beta_j^{-1} )_{j\in\no} \in \ell_{\tn(q_1,q_2)}$, independent of $(M_j)_{j\in\no}$,   
  where for $\tn(q_1,q_2)=\infty$ the space $\ell_\infty$ has to be replaced by $c_0$. 
\end{remark} 

  \begin{remark}
    We would like to give an alternative argument inspired by the proof in \cite{CoDoKu} which works at least in some special cases. The idea is to apply Tong's result \cite{tong} and some factorisation. We sketch it for the
case {$q_1\leq p_1\leq p_2\leq q_2$}. Note that in this situation
    \[
    \frac{1}{\tn(p_1,p_2)}=1-\frac{1}{p_1}+\frac{1}{p_2}\qquad\text{and}\qquad 
    \frac{1}{\tn(q_1,q_2)}=1-\frac{1}{q_1}+\frac{1}{q_2} \ .
    \]
    We decompose
    \[D_\beta: \ell_{q_1}( \ell_{p_1}^{M_j}) \rightarrow 
 \ell_{q_2}\left( \ell_{p_2}^{M_j}\right),\quad D_\beta : (x_{j,m})_{j\in\no, m=1, \dots, M_j} \mapsto (\beta_j^{-1} x_{j,m})_{j\in\no, m=1, \dots, M_j}
    \]
    into
    \[
\begin{array}{ccc}\ell_{q_1}(\ell_{p_1}^{M_j}) & \xrightarrow{~\quad  D_\beta\quad~} &
\ell_{q_2}(\ell_{p_2}^{M_j})
 \\[1ex] D_1 \Big\downarrow & & \Big\uparrow D_2\\[1ex]
\ell_{q_1}(\ell_{q_1}^{M_j}) & \xrightarrow{~\quad  D_0\quad~} &
\ell_{q_2}(\ell_{q_2}^{M_j})
\end{array}
\]    
with
\begin{align*}
%\begin{array}{ll}
  D_1: & \ \ell_{q_1}(\ell_{p_1}^{M_j})  \rightarrow \ell_{q_1}(\ell_{q_1}^{M_j}),\\
  & D_1: (x_{j,m})_{j\in\no, m=1, \dots, M_j} \mapsto \left(M_j^{\frac{1}{p_1}-\frac{1}{q_1}} x_{j,m}\right)_{j\in\no, m=1, \dots, M_j},\\[1ex]
  D_2: & \ \ell_{q_2}(\ell_{q_2}^{M_j})  \rightarrow \ell_{q_2}(\ell_{p_2}^{M_j}),\\
  & D_2: (x_{j,m})_{j\in\no, m=1, \dots, M_j} \mapsto \left(M_j^{\frac{1}{q_2}-\frac{1}{p_2}} x_{j,m}\right)_{j\in\no, m=1, \dots, M_j},\\[1ex]
  D_0: &\ \ell_{q_1}(\ell_{q_1}^{M_j}) \rightarrow \ell_{q_2}(\ell_{q_2}^{M_j}), \\
  & D_0: (x_{j,m})_{j\in\no, m=1, \dots, M_j} \mapsto \left(\beta_j^{-1} M_j^{\frac{1}{\tn(p_1,p_2)}-\frac{1}{\tn(q_1,q_2)}} x_{j,m}\right)_{j\in\no, m=1, \dots, M_j},
%\end{array}
\end{align*}
such that $\ D_\beta = D_2 \circ D_0 \circ D_1$, using also 
$\ \frac{1}{\tn(p_1,p_2)}-\frac{1}{\tn(q_1,q_2)} = 
-\frac{1}{p_1}+\frac{1}{q_1} -\frac{1}{q_2}+\frac{1}{p_2}\ $ in this case. Thus Proposition~\ref{coll-nuc}(iii) leads to
\begin{equation}\label{ddh-1}
\nn{D_\beta} \leq \| D_1\| \ \|D_2\| \ \nn{D_0}.
\end{equation}
We estimate $\|D_1\|$. By our assumption $p_1\geq q_1$ H\"older's inequality leads for $x=(x_{j,m})_{j,m}\in \ell_{q_1}(\ell_{p_1}^{M_j})$ to
\begin{align*}
  \| D_1 x| \ell_{q_1}(\ell_{q_1}^{M_j})\| & =  \Big(\sum_{j=0}^\infty M_j^{(\frac{1}{p_1}-\frac{1}{q_1})q_1} \sum_{m=1}^{M_j} |x_{j,m}|^{q_1} \Big)^{\frac{1}{q_1}} 
  \leq
  %\Big(\sum_{j=0}^\infty M_j^{(\frac{1}{p_1}-\frac{1}{q_1})q_1} \Big(\sum_{m=1}^{M_j} |x_{j,m}|^{p_1}\Big)^{\frac{q_1}{p_1}} M_j^{(\frac{1}{q_1}-\frac{1}{p_1})q_1} \Big)^{\frac{1}{q_1}} \\
  %& =
  \| x | \ell_{q_1}(\ell_{p_1}^{M_j})\|,
\end{align*}
hence $\|D_1\|\leq 1$. Likewise, since $p_2\leq q_2$, by another application of H\"older's inequality,
\begin{align*}
  \| D_2 x| \ell_{q_2}(\ell_{p_2}^{M_j})\| & =  \Big(\sum_{j=0}^\infty M_j^{(\frac{1}{q_2}-\frac{1}{p_2})q_2} \Big(\sum_{m=1}^{M_j} |x_{j,m}|^{p_2}\Big)^{\frac{q_2}{p_2}} \Big)^{\frac{1}{q_2}} %\\
  \leq
  %\Big(\sum_{j=0}^\infty M_j^{(\frac{1}{q_2}-\frac{1}{p_2})q_2} \Big(\sum_{m=1}^{M_j} |x_{j,m}|^{q_2}\Big) M_j^{(\frac{1}{p_2}-\frac{1}{q_2})q_2} \Big)^{\frac{1}{q_2}} \\
  %& =
  \| x | \ell_{q_2}(\ell_{q_2}^{M_j})\|
\end{align*}
for any $x=(x_{j,m})_{j,m}\in \ell_{q_2}(\ell_{q_2}^{M_j})$. Thus $\|D_2\|\leq 1$ and \eqref{ddh-1} implies $\nn{D_\beta} \leq \nn{D_0}$. We would like to use Proposition~\ref{prop-tong}(ii) and have to identify $\ell_{q_r}(\ell_{q_r}^{M_j})$ therefore with $\ell_{q_r}$, $r=1,2$. This can be seen by a bijection like
\begin{align*}
  (x_{j,m})_{j\in\no, m=1, \dots, M_j} \leftrightarrow (y_k)_{k\in\nat}, \quad & y_k=y_{k(j,m)},\ j\in\no, \ m=1, \dots, M_j,\quad \\
  &k(j,m)=\sum_{l=0}^{j-1} M_l + m,
\end{align*}
recall $M_{-1}=0$, i.e., $k(0,m)=m$. Using our previous notation $\alpha_j=\sum_{l=0}^{j-1} M_l$, $j\in\no$, with $\alpha_0=0$, we get $k(j,m)=\alpha_j+m$, $\alpha_{j+1}-\alpha_j=M_j$. For the rewritten sequence $(x_{j,m})_{j\in\no,m=1, \dots M_j}=(y_k)_{k\in\nat}$, $k=k(j,m)$, let $\widetilde{D}_0$ denote the corresponding diagonal operator, acting now as
\[
\widetilde{D}_0 : \ell_{q_1}\rightarrow \ell_{q_2}, \quad \widetilde{D}_0: (y_k)_{k\in\nat} \mapsto \left(\widetilde{\gamma}_k y_k\right)_{k\in\nat},
\]
such that $ D_0 x = \widetilde{D}_0 y$ if $x=(x_{j,m})_{j,m}=(y_k)_{k}=y$ in the above identification. More precisely, using the notation $\gamma_j=\beta_j^{-1} M_j^{\frac{1}{\tn(p_1,p_2)}-\frac{1}{\tn(q_1,q_2)}}$ for the moment, then $\gamma_j x_{j,m} = \widetilde{\gamma}_k y_k$ when $k=k(j,m)$. Consequently, by Prop.~\ref{prop-tong}(ii),
\begin{align*}
\nn{D_\beta}\  & \leq  \nn{D_0: \ell_{q_1}(\ell_{q_1}^{M_j}) \to \ell_{q_2}(\ell_{q_2}^{M_j})}  = \nn{\widetilde{D}_0: \ell_{q_1} \to \ell_{q_2}} \\
  %& = \| \widetilde{\gamma} | \ell_{\tn(q_1,q_2)}\| \\
  &= \Big(\sum_{k=1}^\infty |\widetilde{\gamma}_k |^{\tn(q_1,q_2)}\Big)^{\frac{1}{\tn(q_1,q_2)}} %\\
%&
\quad = \Big(\sum_{j=0}^\infty \sum_{m=\alpha_j+1}^{\alpha_{j+1}} \widetilde{\gamma}_{k(j,m)}^{\tn(q_1,q_2)} \Big)^{\frac{1}{\tn(q_1,q_2)}} \\
&= \Big(\sum_{j=0}^\infty \gamma_j^{\tn(q_1,q_2)} M_j \Big)^{\frac{1}{\tn(q_1,q_2)}} %\\
~  = \| ( \gamma_j M_j^{\frac{1}{\tn(q_1,q_2)}} )_j | \ell_{\tn(q_1,q_2)}\|\\
  &= \left\| (\beta_j^{-1} M_j^{\frac{1}{\tn(p_1,p_2)}} )_j | \ell_{\tn(q_1,q_2)}\right\|
\end{align*}
if $q_2<\infty$. If $q_2=\infty$ and $q_1=1$, then this has to be replaced by $(\beta_j^{-1} M_j^{\frac{1}{\tn(p_1,p_2)}})_j \in c_0$.
  \end{remark}
%%%%%%%%%%%%%%%%%%%%%%%%%%%%%%%%%%%%%%%%%%%%%%%%%%%%%%%%%%%%%%%%%%%

  \section{Nuclear embeddings of   function spaces on domains}\label{sect-fs}

Our second main goal in this paper is to study the nuclearity of Sobolev embeddings acting between function spaces on domains. We briefly recall what is known for bounded Lipschitz domains, and concentrate on quasi-bounded domains afterwards.
  
  \subsection{Embeddings of function spaces on bounded domains}
  In Proposition~\ref{prop-spaces-dom} we have already recalled  the criterion for the compactness of the embedding
  \[ \id_\Omega : \Ae(\Omega) \to \Az(\Omega).\]
Recently Triebel proved in \cite{Tri-nuclear} the following counterpart for its nuclearity.

\begin{proposition}[{\cite{Tri-nuclear,CoDoKu,HaSk-nuc-weight}}]\label{prod-id_Omega-nuc}
  Let $\Omega\subset\rn$ be a bounded Lipschitz domain, $1\leq p_i,q_i\leq \infty$ (with $p_i<\infty$ in the $F$-case), $s_i\in\real$. Then the embedding $\id_\Omega$ given by \eqref{id_Omega} is nuclear if, and only if,
  \begin{equation}
    s_1-s_2 > \nd-\nd\left(\frac{1}{p_2}-\frac{1}{p_1}\right)_+.
\label{id_Omega-nuclear}
  \end{equation}
\end{proposition}

\begin{remark}
  The proposition is stated in \cite{Tri-nuclear} for the $B$-case only, but due to the independence of \eqref{id_Omega-nuclear} of the fine parameters $q_i$, $i=1,2$, and in view of (the corresponding counterpart of) \eqref{B-F-B} it can be extended immediately to $F$-spaces. The if-part of the above result is essentially covered by \cite{Pie-r-nuc} (with a forerunner in \cite{PiTri}). Also part of the necessity of \eqref{id_Omega-nuclear} for the nuclearity of $\id_\Omega$ was proved by Pietsch in \cite{Pie-r-nuc} such that only the limiting case $ s_1-s_2 = \nd-\nd(\frac{1}{p_2}-\frac{1}{p_1})_+$ was open for many decades. Only recently Edmunds, Gurka and Lang in \cite{EGL-3} (with a forerunner in \cite{EL-4}) obtained some answer in the limiting case which was then completely solved in \cite{Tri-nuclear}. 
  Note that in \cite{Pie-r-nuc} some endpoint cases (with $p_i,q_i\in \{1,\infty\}$) have already been discussed for embeddings of Sobolev and certain Besov spaces (with $p=q$) into Lebesgue spaces. In our recent paper \cite{HaSk-nuc-weight} we were able to further extend Proposition~\ref{prod-id_Omega-nuc} in view of the borderline cases.

%\begin{remark}\label{rem-p*-tong}
  For better comparison one can reformulate the compactness and nuclearity characterisations of $\id_\Omega$ in \eqref{id_Omega-comp} and \eqref{id_Omega-nuclear} as follows, involving the number $\tn(p_1,p_2)$ defined in \eqref{tongnumber}. Let $1\leq p_i,q_i\leq \infty$, $s_i\in\real$ and 
  \[\delta= s_1 - \frac{\nd}{p_1}-s_2 + \frac{\nd}{p_2}.\]
   Then
  \begin{align*}
    \id_\Omega: \Ae(\Omega) \to \Az(\Omega) \quad  \text{is compact}\quad & \iff \quad \delta> \frac{\nd}{p^\ast}\qquad\text{and}\\
   \id_\Omega: \Ae(\Omega) \to \Az(\Omega) \quad \text{is nuclear}\quad & \iff \quad \delta > \frac{\nd}{\tn(p_1,p_2)}.
  \end{align*}
  Hence apart from the extremal cases $\{p_1,p_2\}=\{1,\infty\}$ (when $\tn(p_1,p_2)=p^\ast$) nuclearity is indeed stronger than compactness, i.e.,
  \[
  \id_\Omega: \Ae(\Omega) \to \Az(\Omega) \] %\quad
  \text{is compact, but not nuclear, \quad if, and only if, }%\quad \iff \quad
  \[\frac{\nd}{p^\ast} < \delta \leq \frac{\nd}{\tn(p_1,p_2)}.
  \]
  We observed similar phenomena in the weighted setting in \cite{HaSk-nuc-weight}, {recall also our discussion in Remark~\ref{rem-extremal-seq} for the sequence space situation. }
 % \end{remark}

\end{remark}
  
\begin{remark}
  In \cite{CoDoKu} the authors dealt with the nuclearity of the embedding $B^{s_1,\alpha_1}_{p_1,q_1}(\Omega)\to B^{s_2,\alpha_2}_{p_2,q_2}(\Omega)$ where the indices $\alpha_i$ represent some  additional logarithmic smoothness. They obtained a characterisation for almost all possible settings of the parameters. Finally, 
in \cite{CoEdKu} some further limiting endpoint situations of nuclear embeddings like $\id:B^{\nd}_{p,q}(\Omega)\to L_p(\log L)_a(\Omega)$ are studied. For some weighted results see also \cite{Parfe-2} and our recent contribution \cite{HaSk-nuc-weight}.
\end{remark}

  %%%%%%%%%%%%%%%%%%%%%%%%%
  \subsection {Embeddings of function spaces on quasi-bounded domains}
Now we study so called quasi-bounded domains and need to introduce them first, together with their wavelet characterisation. % similar to that in Theorem \ref{waveweight}.
An unbounded  domain $\Omega$ in $\R^d$  is called {\em quasi-bounded} if
$$
\lim_{x\in \Omega, |x|\rightarrow \infty} \dist(x, \partial \Omega)\, =\,  0~~.
$$
An unbounded domain is not quasi-bounded if, and only if, it contains infinitely many pairwise disjoint congruent balls, cf. \cite{AF},  page 173. \\

Before we can formulate the properties of embeddings of function spaces defined on quasi-bounded domains, we first need to extend the notion of $\A(\Omega)$ as given in Section~\ref{sect-1}. Here we  follow the ideas of Triebel in \cite{T08}  and define  spaces $\bar{F}^s_{p,q}(\Omega)$ and $\bar{B}^s_{p,q}(\Omega)$.

 We put	
\[
	\sigma_p\, =\, d\left(\frac{1}{p} - 1 \right)_+\quad\text{and}\quad \sigma_{p,q}  =  d\left(\frac{1}{\min (p,q) } - 1 \right)_+ , \quad 0 < p,q \le \infty.
\]

\begin{definition}\label{function spaces} %[cf.\cite{T08}]
Let $\Omega$ be an arbitrary domain in $\R^d$ with $\Omega \not= \R^d$ and let
	\[ 0< p \le \infty , \qquad  0< q \le \infty , \qquad  s\in \R ,
\]
with $p<\infty$ for the $F$-spaces.
\benu[\bfseries\upshape (i)]
\item
Let
\begin{align*}
%	\label{Atilde2}\nonumber
  \widetilde{A}^s_{p,q}(\Omega)\, = \,\Big\{ f\in \cD'(\Omega): \ \exists\ g\in \A(\rn)\ \text{with}\ g|_\Omega = f\ \text{and}\ \supp g\subset  \overline{\Omega}
  %:\quad f= g|_\Omega \; \text{for some}\; g\in {A}^s_{p,q}(\R^d), \; \supp g\subset  \overline{\Omega}
  \Big\},
\end{align*}
equipped with the quotient norm,
\begin{equation}
	\label{Atilde3}\nonumber
	\|f|\widetilde{A}^s_{p,q}(\Omega)\|\, = \,\inf \| g|A^s_{p,q}(\R^d) \|,
\end{equation}
where the infimum is taken over all $g\in {A}^s_{p,q}(\overline{\Omega})$ with $f= g|_\Omega$ .

\item
We define
\begin{equation}
	\label{Fbar}\nonumber
	\bar{F}^s_{p,q}(\Omega)\,=\,
	\begin{cases}
	\widetilde{F}^s_{p,q}(\Omega)& \text{if}\quad 0< p < \infty, \;0 < q \le \infty, \; s > \sigma_{p,q}\, , \\
	{F}^0_{p,q}(\Omega)& \text{if}\quad 1<p<\infty, \;1 \le q\le \infty, \; s = 0\, , \\
	{F}^s_{p,q}(\Omega)& \text{if}\quad 0<p<\infty, \;0 < q \le \infty, \; s < 0 \, ,
	\end{cases}
\end{equation}
and
\begin{equation}
	\label{Bbar}\nonumber
	\bar{B}^s_{p,q}(\Omega)\,=\,
	\begin{cases}
	\widetilde{B}^s_{p,q}(\Omega)& \text{if}\quad 0< p \le \infty, \;0 < q \le \infty , \; s > \sigma_{p}\, , \\
	{B}^0_{p,q}(\Omega)& \text{if}\quad 1<p<\infty, \;0 < q\le \infty , \; s = 0\, , \\
	{B}^s_{p,q}(\Omega)& \text{if}\quad 0<p<\infty, \;0 < q \le \infty , \; s < 0 \, .
	\end{cases}
\end{equation}
\eenu
\end{definition}

Next we make use of some quantities describing the  quasi-boundedness of the domain. For that reason  we introduced in \cite{Leo-Sk-2} {\em a box packing number} $b(\Omega)$ of an open set $\Omega$. We recall the definition here. 

Let $Q_{j,m}$ denote the dyadic cube in $\R^d$  with side-length $2^{-j}$, $j\in \N_0$, given by %whose vertices are adjacent points of the lattice $2^{-j}m$  and $m\in \Z^d$.   More precisely we put
\[ 
Q_{j,m}\, =\, [0,2^{-j})^d + 2^{-j}m ,\qquad m\in \Z^d,\; j\in \N_0\, . 
\] 
Let $\Omega\subset \R^d$ be a non-empty  open set with $\Omega \not= \R^d$. For $j\in\No$ we denote by 
\begin{align*}%\label{box1}
  b_j(\Omega) = \sup \big\{ k\in\nat : & \bigcup_{\ell=1}^k Q_{j,m_\ell}\subset \Omega,\\% \quad
  &Q_{j,m_\ell}\; \text{being %pairwise disjoint 
	dyadic cubes of side-length}\; 2^{-j} \big\}.
\end{align*}
If there is no dyadic cube of size $2^{-j}$ contained in $\Omega$ we put $b_j(\Omega)=0$. The following properties of the sequence $\big(b_j(\Omega)\big)_{j\in\No}$ are  obvious:
\benu[(i)]
%\item $0\le b_j(\Omega)\le \infty$ for any $j\in \N_0$ and $0< b_j(\Omega)$ for sufficiently large $j$.
\item There exists a constant $j_0=j_0(\Omega)\in \N_0$ %and $c_o=c_o(\Omega)>0$
such that for any $j\ge j_0$ we have
\begin{equation}\label{box2}
	0 < 2^{d(j-j_0)} b_{j_0}(\Omega)\le b_j(\Omega) . 
\end{equation}
\item If $|\Omega|<\infty$, then
\begin{equation}\label{box3}
b_j(\Omega) 2^{-jd} \le |\Omega| \, .
\end{equation}
\eenu
It follows from \eqref{box2} that if $0<s< d$, then $\lim\limits_{j\rightarrow\infty}b_j(\Omega) 2^{-js}=\infty$.
Moreover, if $s_1<s_2$ and the sequence $b_j(\Omega) 2^{-js_1}$ is bounded, then $\lim\limits_{j\rightarrow\infty}b_j(\Omega) 2^{-js_2}=0$.
Thus there exists at most one number $b\in \R$ such that $\limsup\limits_{j\rightarrow\infty}b_j(\Omega) 2^{-js}=\infty$ if $s<b$, and $\lim\limits_{j\rightarrow\infty}b_j(\Omega) 2^{-js}=0$ if $s>b$. We put
\begin{equation}\label{box4}
b(\Omega)= \sup\big\{t\in \R_+:  \limsup_{j\rightarrow\infty}b_j(\Omega) 2^{-jt}=\infty \big\}.
\end{equation}

\begin{remark}
For any non-empty open set $\Omega\subset \R^d$ we have $d\le b(\Omega)\le \infty$. If $\Omega$ is unbounded and not quasi-bounded, then $b(\Omega)=\infty$. But there are also quasi-bounded domains such that   $b(\Omega)=\infty$, cf. \cite{Leo-Sk-2}.  %Example \ref{example1aa} below.  
 Moreover it follows from \eqref{box2} and \eqref{box3} that if the measure $|\Omega|$ is finite, then $b(\Omega)=d$.
\end{remark}

%%%%%%%%%%%%%%%%%%%%%%
To illustrate this definition we  give  simple examples of quasi-bounded domains. 

\begin{example}\label{example3}{\rm 
Let $\alpha>0$ and $\beta>0$ . 
We consider the open sets $\omega_\alpha,\omega_{1,\beta}\subset \R^2$ defined as follows
\begin{eqnarray}
\omega_\alpha &=& \{(x,y)\in \R^2: \; |y|< x^{-\alpha}, x > 1 \}\ , \nonumber\\ 
%\qquad \text{and}\qquad\Omega_\alpha = \{(x,y)\in \R^2: \; |y|< |x|^{-\alpha} \}\, ,\\
\omega_{1,\beta} &=& \{(x,y)\in \R^2: \; |y|< x^{-1}(\log{x})^{-\beta}, \,x > e \} \nonumber .
%\qquad \text{and}\qquad\Omega_\alpha = \{(x,y)\in \R^2: \; |y|< |x|^{-\alpha} \}\, . \nonumber
\end{eqnarray}
One can easily calculate that
\begin{align}%\begin{eqnarray}
b_j(\omega_\alpha)\, \sim \,
\begin{cases}
2^{j(\alpha^{-1}+1)} & \text{if}\quad 0<\alpha< 1\, ,\\
j 2^{2j} & \text{if}\quad \alpha = 1\, ,\\
2^{2j}& \text{if}\quad \alpha > 1\, ,
\end{cases} \nonumber
\intertext{and}
b_j(\omega_{1,\beta})\, \sim \,
\begin{cases}\nonumber
  2^{2j} \, j^{1-\beta} & \text{if}\quad \beta < 1\, ,\\
2^{2j} \,\log{j} & \text{if}\quad \beta = 1\, ,\\
 2^{2j}& \text{if}\quad \beta > 1\, .
\end{cases}
\end{align}%\end{eqnarray}
In consequence
\[
b(\omega_\alpha)\, = \,
\begin{cases}
\alpha^{-1}+1 & \text{if}\quad 0<\alpha< 1\, ,\\
2             & \text{if}\quad \alpha \ge  1\, .
\end{cases}
\]
Moreover, the limit $\lim_{j\rightarrow \infty} b_j(\omega_\alpha)2^{-jb(\omega_\alpha)}$ is a positive  finite number if $\alpha \not= 1$. But if $\alpha = 1$, then the limit equals infinity.

In a similar way  
$
b(\omega_{1,\beta})\, = \, 2 $ \mbox{ for all } $\beta$ and $\lim_{j\rightarrow \infty} b_j(\omega_{1,\beta})2^{-jb(\omega_{1,\beta})} = \infty $ if $0<\beta\le 1$.
}
\end{example}
%%%%%%%%%%%%%%%%%
\begin{remark}\label{remquasib}
  \begin{enumerate}[(i)]
\item One can construct a quasi-bounded domain in $\R^d$ with prescribed sequence $b_j(\Omega)$, cf. \cite{Leo-Sk-2}. Some more concrete examples based on this general construction can be found in  \cite[Section 3]{Leo-Sk-3}.
\item Another characterisation of $b(\Omega)$  was proved in  \cite{Leo-Sk-3}. For any domain $\Omega\not= \R^d $ and any $r>0$ we put 
\begin{equation}\label{omegar}
	\Omega_r \, =\, \{ x\in \Omega: \;  \dist(x, \partial \Omega) > r \} . 
\end{equation}
If the domain $\Omega $ is quasi-bounded,  then $|\Omega_r|<\infty$ for any $r>0$ and 
\begin{equation}\label{omegab}
	b(\Omega)\, =\, d + \limsup_{r\rightarrow 0} \left| \frac{\log_2 |\Omega_r|}{\log_2 r}\right| .
\end{equation}
\end{enumerate}\end{remark}

To give the wavelet characterisation of the spaces  $\bar{B}^s_{p,q}(\Omega)$ and $\bar{F}^s_{p,q}(\Omega)$ we need some additional assumptions concerning the underlying domain $\Omega$. Namely we should assume that the domain is 
E-thick (exterior thick) and E-porous, cf. \cite[Chapter 3]{T08}.  Now we recall the definition starting  with porosity.

\begin{definition}\benu[\bfseries\upshape (i)]
  \item
A closed set $\Gamma\subset \R^d$ is said to be porous if there exists a number $0<\eta<1$ such that one  finds for any ball $B(x,r)\subset \R^d$ centered at $x $ and of radius $r$ with $0<r<1$,  a ball $B(y, \eta r)$ with
\[ B(y,\eta r) \subset B(x, r)\qquad \text{and}\qquad B(y,\eta r) \cap \Gamma = \emptyset\, .
\]
\item  A closed set $\Gamma\subset \R^d$ is said to be uniformly porous if it is porous and there is a locally finite positive Radon measure $\mu$ on $\R^d$ such that $\Gamma = \supp \mu$ and
\[ \mu(B(\gamma, r))\, \sim \, h(r)\, ,\qquad \text{with}\qquad  \gamma\in \Gamma, \quad 0<r<1\, , \]
where $h:[0,1]\rightarrow \R$ is a continuous strictly increasing function with $h(0)=0$ and $h(1)=1$ (the equivalence constants are independent of $\gamma$ and $r$).
\eenu
\end{definition}

\begin{remark}
The closed set $\Gamma$ is called an ${\alpha}$-set if there exists a locally finite positive Radon measure $\mu$ on $\R^d$ such that $\Gamma = \supp \mu$ and
\[ \mu(B(\gamma, r))\, \sim \, r^\alpha\, ,\qquad \text{with}\qquad  \gamma\in \Gamma, \quad 0<r<1\, . \]
Naturally $0\le \alpha\le d$.  Any $\alpha$-set with $\alpha<d$ is  uniformly porous.
\end{remark}

\begin{definition}
Let $\Omega$ be an open set in $\R^d$ such that $\Omega\not= \R^d$ and $\Gamma=\partial \Omega$.
\benu[\bfseries\upshape (i)]
\item The domain $\Omega$ is said to be E-thick if one can find for any interior cube $Q^i \subset \Omega$ with
\[ \ell(Q^i) \,\sim\, 2^{-j}, \quad\text{and}\quad \dist(Q^i,\Gamma)\,\sim\, 2^{-j}\, ,\quad j\ge j_0\in \N, \]
a complementing exterior cube
$Q^e \subset \R^d\setminus \Omega$ with
\[ \ell(Q^e) \,\sim\, 2^{-j}, \quad\text{and}\quad \dist(Q^e,\Gamma)\,\sim\, \dist(Q^e,Q^i)\,\sim\,2^{-j}\, ,\quad j\ge j_0\in \N\, . \]
$Q^i$ and $Q^e$ denote  cubes in $\R^d$ with sides parallel to the axes of coordinates. Moreover $\ell(Q)$ denotes the side-length of the cube $Q$. 
\item The domain $\Omega$ is said to be E-porous if there is a number $\eta$ with $0<\eta<1$ such that  one  finds for any ball $B(\gamma ,r)\subset \R^d$ centred  at $\gamma\in \Gamma$ and of radius $r$ with $0<r<1$,  a ball $B(y, \eta r)$ with
\[ B(y,\eta r) \subset B(\gamma, r)\qquad \text{and}\qquad B(y,\eta r) \cap \overline{\Omega} = \emptyset\, .
\]
\item
The domain $\Omega$ is called uniformly E-porous if it is E-porous and $\Gamma$ is uniformly porous.
\eenu
\end{definition}

\begin{remark}
We collect some observations. 
  \benu[(i)]
 \item If $\Omega$ is E-porous, then $\Omega $ is E-thick and $|\Gamma|=0$. On the other hand,
if $\Omega$ is E-thick and $\Gamma$ is an $\alpha$-set, then $\Omega$ is uniformly E-porous and $d-1\le \alpha < d$.
\item There are quasi-bounded domains that are not E-porous or even not E-thick, cf. e.g. \cite{AF},  page 176 for the example of a quasi-bounded domain with empty exterior. 
\item The  domains given in Example \ref{example3} and pointed out in Remark~\ref{remquasib} are not only quasi-bounded but also uniformly E-porous.
\eenu
\end{remark}

If the domain is uniformly E-porous, then one can characterise $\bar{A}^s_{p,q}(\Omega)$ spaces in terms of the wavelet expansion of the distributions. %cf. Section 1.3 below.
Now we give the wavelet characterisation of the spaces  $\bar{F}^s_{p,q}(\Omega)$ and $\bar{B}^s_{p,q}(\Omega)$.
 Let $N_j \in \overline{\N}$, $j\in \N_0$ and $\overline{\N}= \N\cup\{\infty \}$.  %We will work with the following sequence spaces
\begin{eqnarray*}
\ell_q \big( 2^{j\sigma} \ell_p^{N_j}\big)   & := &   \Big\{ \lambda =
(\lambda_{j,k})_{j\in \N_0,k=1,\cdots,N_j} : \quad     \lambda_{j,k} \in \C\, , \\
&&  \qquad
\big\| \, \lambda \, \big|\ell_q \big( 2^{j\sigma} \ell_p^{N_j}\big)\big\| =
 \Big\| \Big(2^{j\sigma}\,
\big( \sum_{k =1}^{N_j}|\lambda_{j,k}\,|^p
\bigg)^{1/p}\Big)_{j=0}^{\infty} \big| \ell_q \Big\| < \infty \Big\}
\end{eqnarray*}
(usual modifications if $p=\infty$ and/or $q=\infty$). If $N_j=\infty$, then $\ell_p^{N_j}= \ell_p$.

\begin{theorem}\label{wavelet-th} Let $\Omega$ be a uniformly E-porous domain in $\R^d$, $\Omega\not= \R^d$. {Let $s\in\real$, $0<p,q\leq\infty$, and $\bar{B}^s_{p,q}(\Omega)$ be defined as in Definition~\ref{function spaces}.} 
 Let $u\in \N_0$. 
Then there exists  an orthonormal basis
\[
 \left\{\Phi^j_r:\; j\in \N_0;\; r=1,\ldots ,N_j\right\},\qquad
 \Phi^j_r \in C^u (\Omega), \quad j\in \N_0, \;r=1,\ldots ,N_j,   
  %\qquad \text{with}\qquad \supp \Phi^j_r \subset B(x_r^j, c_2 2^{-j}), \quad j\in \N_0
\]
%be a  u-wavelet system that is an orthonormal basis 
in $L_2(\Omega)$,   such that  if $u>\max(s,\sigma_p-s)$,
then $f\in \cD'(\Omega)$ is an element of $	\bar{B}^s_{p,q}(\Omega)$  if, and only if, it can be represented as
\begin{equation}\label{wavelet-rep}
	f\,=\, \sum_{j=0}^\infty\sum_{r=1}^{N_j} \lambda_r^j2^{-jd/2} \Phi_r^j, \qquad \lambda\in \ell_q \big( 2^{j(s-\frac{d}{p})} \ell_p^{N_j}\big),
\end{equation}
unconditional convergence being in $\cD'(\Omega)$. 

Furthermore, if $f\in \bar{B}^s_{p,q}(\Omega)$, then the representation \eqref{wavelet-rep} is unique with $\lambda = \lambda(f)$
\begin{equation}\label{wavelet_coef}\nonumber
\lambda_r^j\, =\, 	\lambda_r^j(f)\,=\,  2^{jd/2} (f, \Phi_r^j),
\end{equation}
where $(\cdot , \cdot )$ is a dual pairing and
\begin{equation}\label{wavelet-iso}\nonumber
I:\,\bar{B}^s_{p,q}(\Omega)\ni  f \mapsto \lambda(f) \in \ell_q \big( 2^{j(s-\frac{d}{p})} \ell_p^{N_j}\big)
\end{equation}
is an isomorphism. If in addition $\max{(p,q)}<\infty$\,, then $\left\{\Phi^j_r\right\}$ is an unconditional basis in  $\bar{B}^s_{p,q}(\Omega)$.
\end{theorem}

\begin{remark}
The above theorem was proved by Triebel, cf. \cite[Theorem 3.23]{T08}. He used so called $u$-wavelet systems,  cf. \cite[Chapter 2]{T08} for the construction  of this wavelet system.  The sketch of the construction can be found in \cite{Leo-Sk-2}.  If we assume that the domain $\Omega$ is only E-thick, then the theorem holds for  $\bar{B}^s_{p,q}(\Omega)$ with $s\not= 0$ \cite[Theorem 3.13]{T08}. Similar results were obtained for $\bar{F}^s_{p,q}(\Omega)$ spaces.
\end{remark}

There is a strict relation between the numbers  $N_j$ used in the last theorem and the sequence $b_j(\Omega)$. Namely it was proved in \cite{Leo-Sk-2} that 
\begin{equation}\label{box10}
	2^{d}b_{j-2}(\Omega) \le N_j \le b_j(\Omega)\, .
\end{equation}

%END NEW
\bigskip
Using the above wavelet characterisation we obtained  the necessary and sufficient conditions for  compactness of embeddings of the function spaces defined on a uniformly E-porous quasi-bounded domain. We refer to \cite{Leo-Sk-2} for  the proof. 

\begin{proposition}\label{emb2}
Let $\Omega$ be a uniformly E-porous quasi-bounded domain in $\R^d$  and let  $s_1 > s_2$, {$0<p_i,q_i\leq\infty$  (with $p_i<\infty$ in case of $A=F$), $i=1,2$, and the spaces $\bar{A}^s_{p,q}(\Omega)$ be given by Definition~\ref{function spaces}}. 
\benu[\bfseries\upshape (i)]
\item
If  $b(\Omega)=\infty$\,, then the  embedding
\begin{equation}\label{boxemb}
\bar{A}^{s_1}_{p_1,q_1}(\Omega) \,\hookrightarrow\, \bar{A}^{s_2}_{p_2,q_2}(\Omega)
\end{equation}
is compact
if, and only if, $p_1\le p_2$  and
\begin{equation}\label{boxemb2}
s_1-\frac{d}{p_1}-s_2 + \frac{d}{p_2} >  0\, . %\qquad \text{if}\qquad q^*<\infty  .
\end{equation}
\item
Let  $b(\Omega)<\infty$.  The embedding
\begin{equation}\label{boxemb-a}
\bar{A}^{s_1}_{p_1,q_1}(\Omega) \,\hookrightarrow\, \bar{A}^{s_2}_{p_2,q_2}(\Omega)
\end{equation}
 is compact if
\begin{equation}
\label{box7}
	s_1-s_2 - d \Big(\frac{1}{p_1} - \frac{1}{p_2}\Big)  > \frac{b(\Omega)}{p^*}~. %\qquad\text{for}\qquad %q^*<\infty
\qquad
\end{equation}

If the embedding \eqref{boxemb-a} is compact and $p^\ast=\infty$, %$\frac{1}{p^*}=0$,
then $s_1-s_2 - d(\frac{1}{p_1} - \frac{1}{p_2})>0$.

If the embedding \eqref{boxemb-a} is compact and $p^\ast<\infty$, %$\frac{1}{p^*}>0$,
then $s_1-s_2 - d(\frac{1}{p_1} - \frac{1}{p_2})\ge \frac{b(\Omega)}{p^*}$.
\eenu
\end{proposition}

\begin{remark} We collect some results about the sharpness of the above statement.
  \begin{enumerate}[(i)]
\item In case of $b(\Omega) < \infty$  one can prove that \eqref{box7} is a sufficient and necessary condition for  compactness of the embeddings, except of the case $p^\ast<\infty $ and $\limsup_{j\rightarrow \infty} b_j(\Omega)2^{-jb(\Omega)}=0$.
\item If the domain $\Omega$ is not quasi-bounded, then the embedding  \eqref{boxemb-a} is never compact, cf. \cite{Leo-Sk-2}. 
\item  If  $\Omega$ is a domain in $\R^d$ with finite Lebesgue measure, then the embedding  \eqref{boxemb-a} is   compact if, and only if,
\begin{eqnarray*}
	s_1-s_2 - \Big(\frac{d}{p_1} - \frac{d}{p_2}\Big)_+  >  0  \,  .
\end{eqnarray*}
cf. \cite{Leo-Sk-2}.   Thus for a set of finite Lebesgue measure we get the same conditions for  compactness as for bounded smooth domains, {recall Proposition~\ref{prop-spaces-dom}}.  
\item
 On the other hand, if the domain is not quasi-bounded,  then the Sobolev embeddings are never compact. So the most interesting case are the quasi-bounded domains with infinite measure. If $\Omega$ is such a domain,  then all numbers $b_j(\Omega)$ are finite. But in contrast to the  domain with a finite measure, the numbers $b_j(\Omega)$ are not asymptotically equivalent to $2^{jd}$.
 \item { One can use the number $b(\Omega)$ to describe quantitative properties of corresponding compact embeddings  in terms of the asymptotic behaviour of their  $s$-numbers and entropy numbers, cf. \cite{DZ,Leo-Sk-2, Leo-Sk-3}.  }
\end{enumerate} 
\end{remark}

\begin{theorem}\label{nuclear-quasi}
Let $\Omega$ be a uniformly E-porous quasi-bounded domain in $\R^d$ and  let $s_1 > s_2$. Assume $1\leq p_i,q_i\leq\infty$, $i=1,2$. %  (with $p_i<\infty$ in the $F$-case), $i=1,2$. 
\benu[\bfseries\upshape (i)]
%\begin{idemize}
%\item{(i)} If  and let $b(\Omega) < \infty$ .
\item If  $b(\Omega)=\infty$\,, then the  embedding
\begin{equation}\label{boxemb11}
\id^B_\Omega : \bar{B}^{s_1}_{p_1,q_1}(\Omega) \,\hookrightarrow\, \bar{B}^{s_2}_{p_2,q_2}(\Omega)
\end{equation}
is nuclear
if, and only if, $p_1=1$,  $p_2=\infty$  and $s_1-s_2>d$. 
%\begin{equation}\label{boxemb22}
%s_1-\frac{d}{p_1}-s_2 + \frac{d}{p_2} >  0\, . %\qquad \text{if}\qquad q^*<\infty  .
%\end{equation}
\item
Let  $b(\Omega)<\infty$.   The embedding $\ \id^B_\Omega \ $ given by \eqref{boxemb11} 
%\begin{equation}\label{boxemb-aa}
%\id^B_\Omega : \bar{B}^{s_1}_{p_1,q_1}(\Omega) \,\hookrightarrow\, \bar{B}^{s_2}_{p_2,q_2}(\Omega)
%\end{equation}
 is nuclear if
\begin{equation}
\label{box77}
	s_1-s_2 - d \Big(\frac{1}{p_1} - \frac{1}{p_2}\Big)  > \frac{b(\Omega)}{\tn(p_1,p_2)}~. %\qquad\text{for}\qquad %q^*<\infty
\qquad
\end{equation}
Conversely, if the embedding \eqref{boxemb11} is nuclear and $\tn(p_1,p_2)=\infty$, %$\frac{1}{\tn(p_1,p_2)}=0$,
that is, $p_1 = 1$ and $p_2=\infty$, then $s_1-s_2 - d(\frac{1}{p_1} - \frac{1}{p_2})>0$.

If the embedding \eqref{boxemb11} is nuclear and $\tn(p_1,p_2)<\infty$, % $\frac{1}{\tn(p_1,p_2)}>0$,
then $s_1-s_2 - d(\frac{1}{p_1} - \frac{1}{p_2})\ge \frac{b(\Omega)}{\tn(p_1,p_2)}$.
\eenu
\end{theorem}

\begin{proof}

  {\em Step 1.}~
%  Note first that it is sufficient to consider $B$-spaces only, as the counterpart for $F$-spaces is due to the (adapted version of the) embeddings \eqref{B-F-B} and the independence of all the above conditions from $q_1,q_2$. 
  At first we show that 
  \begin{align}
    \label{nuc-qb-0}
    \id^B_\Omega\quad\text{is nuclear} & \quad \text{if, and only if,}\\%\quad \\
    \nonumber & \begin{cases} 
    (2^{-j\delta} b_j(\Omega)^{\frac{1}{\tn(p_1,p_2)}})_{j\in\no} \in \ell_{\tn(q_1,q_2)} &\text{if}\ \tn(q_1,q_2)<\infty, \\  (2^{-j\delta} b_j(\Omega)^{\frac{1}{\tn(p_1,p_2)}})_{j\in\No}\in c_0 & \text{if}\ \tn(q_1,q_2)=\infty. \end{cases}
  \end{align}
Recall $ \delta = (s_1-\frac{d}{p_1})-(s_2-\frac{d}{p_2})$. To verify this, we consider the following diagram:
\[
\begin{array}{ccccc} \bar{B}^{s_1}_{p_1,q_1}(\Omega) & \xrightarrow{~\quad  I_1 \quad~} & \ell_{q_1}(2^{j(s_1-\frac{d}{p_1})}\ell_{p_1}^{N_j}) & \xrightarrow{~\quad  D_1~} &
\ell_{q_1}(\ell_{p_1}^{N_j})
 \\[1ex] \id^B_\Omega \Big\downarrow & &&& \Big\downarrow D_\lambda\\[1ex]
\bar{B}^{s_2}_{p_2,q_2}(\Omega)&\xleftarrow{~\quad  I_2^{-1} \quad~}  &\ell_{q_2}(2^{j(s_2-\frac{d}{p_2})}\ell_{q_2}^{N_j}) & \xleftarrow{~\quad  D^{-1}_2\quad~} &
\ell_{q_2}(\ell_{{p_2}}^{N_j})
\end{array}
\]    
with
$I_1$ and $I_2^{-1}$ denote the wavelet-isomorphism from Theorem \ref{wavelet-th} and for $i = 1 , 2$ we define
\begin{align*}
%\begin{array}{ll} 
  & D_i:  \ \ell_{q_i}(2^{j(s_i-\frac{d}{p_i})}\ell_{p_i}^{N_j}) \rightarrow \ell_{q_i}(\ell_{q_i}^{N_j}),\\
  &\qquad\quad~ D_i: (x_{j,m})_{j\in\no, m=1, \dots, N_j}\mapsto \left( 2^{j(s_i-\frac{d}{p_i})}x_{j,m}\right)_{j\in\no, m=1, \dots, N_j},\\[1ex]
%& x_{j,m} = \Bigg\{
%\begin{array}{lcl} 2^{j(s_1-\frac{d}{p_1})}  \lambda_{j,m} & if & m = 1, \cdots ,N_j~~j\in\N_0 \\
%0  &if &m = N_j + 1, \cdots ,b_j(\Omega) ~~j\in\N_0
%\end{array} 
%\\[1ex]
  & D^{-1}_i:  \ \ell_{q_i}(\ell_{{p_i}}^{N_j}) \rightarrow \ell_{q_i}(2^{j(s_i-\frac{d}{p_i})}\ell_{p_i}^{N_j}), \\
  &\qquad\quad~ D^{-1}_i: (x_{j,m})_{j\in\no, m=1, \dots, N_j} \mapsto \left(  2^{-j(s_i-\frac{d}{p_i})}x_{j,m}\right)_{j\in\no, m=1, \dots, N_j},\\[1ex]
  & D_\lambda:  \ \ell_{q_1}(\ell_{{p_1}}^{N_j}) \rightarrow \ell_{q_2}(\ell_{{p_2}}^{N_j}), \\
  &\qquad\quad~ D_\lambda: (x_{j,m})_{j\in\no, m=1, \dots, N_j} \mapsto \left( 2^{j((s_2-\frac{d}{p_2})-(s_1-\frac{d}{p_1}))} x_{j,m}\right)_{j\in\no, m=1, \dots, M_j},
%\end{array}
\end{align*}
such that $\ \id^B_\Omega = I_2^{-1} \circ  D^{-1}_2 \circ D_\lambda \circ D_1 \circ I_1  \ $ and, vice versa,  $\ D_\lambda = D_2 \circ I_2 \circ \id^B_\Omega \circ I^{-1}_1 \circ D^{-1}_1$.
Then Proposition \ref{coll-nuc}(iii) applied to $\id^B_\Omega$ and $D_\lambda$, respectively, implies that the embedding $\id^B_\Omega$ is nuclear if, and only if, the operator $D_\lambda$ is nuclear, which -- in view of Proposition \ref{main-seq-nuc} -- is the case, if, and only if, 
the sequence $\lambda$ with $\lambda_{j,m} = 2^{-j\delta}$ for $j\in\no$, $m=1, \dots, N_j$, belongs to  $\ds \mathfrak{l}_{\tn(q_1,q_2)}(\mathfrak{l}^{N_j}_{\tn(p_1,p_2)})$, replaced by $c_0$ if $\tn(q_1,q_2)=\infty$.  But by definition of the spaces,
\begin{equation}\label{nuc-qb-1}
  \|\lambda
| \mathfrak{l}_{\tn(q_1,q_2)}(\mathfrak{l}^{N_j}_{\tn(p_1,p_2)})\| = \Big\| \big(2^{-j\delta} N_j^{\frac{1}{\tn(p_1,p_2)}}\big)_{j\in\no} | \ell_{\tn(q_1,q_2)}\Big\|,
\end{equation}
so $\id^B_\Omega$ is nuclear, if, and only if, $(2^{-j\delta} N_j^{\frac{1}{\tn(p_1,p_2)}})_{j\in\No}\in \ell_{\tn(q_1,q_2)}$, replaced by $(2^{-j\delta} N_j^{\frac{1}{\tn(p_1,p_2)}})_{j\in\No}\in c_0$ if $\tn(q_1,q_2)=\infty$. 

%Conversely, if we assume that $\id_\Omega$ is nuclear, then again by Proposition \ref{coll-nuc}(iii) also $D_\lambda$ becomes nuclear and consequently  $ \|\lambda| X_{\tn(p_1,p_2),\tn(q_1,q_2)}^{N}\| $ has to be finite. Again using \eqref{nuc-qb-1} we can thus conclude that $\id_\Omega$ is nuclear if, and only if,  $(2^{-j\delta} N_j^{\frac{1}{\tn(p_1,p_2)}})_{j\in\no} \in \ell_{\tn(q_1,q_2)}$, with $(2^{-j\delta} N_j^{\frac{1}{\tn(p_1,p_2)}})_{j\in\No}\in c_0$ if $\tn(q_1,q_2)=\infty$. 

Because of $2^d b_{j-2}(\Omega) \le N_j \le b_j(\Omega)$ - see \eqref{box10} -  we can replace  $N_j$  by $b_j(\Omega)$ in the definition of $\lambda$, and finally arrive at \eqref{nuc-qb-0}.

In particular, if  $p_1=1$, $p_2=\infty$, that is, $\tn(p_1,p_2)=\infty$, then $\delta = s_1-s_2-d$, so $\id^B_\Omega$ is nuclear if, and only if, $s_1-s_2>d$, which  completes the proof in case of (i) and (ii) in this setting. It remains to deal with the remaining cases for $\tn(p_1,p_2)<\infty$ in dependence on $b(\Omega)$. \\

{\em Step 2.}~ Next we assume that  $b(\Omega)<\infty$ and $\tn(p_1,p_2)<\infty$. Let   $\delta > \frac{b(\Omega)}{\tn(p_1,p_2)}$ and choose $s>b(\Omega)$ such that $\delta > \frac{s}{\tn(p_1,p_2)} >\frac{b(\Omega)}{\tn(p_1,p_2)}$.   Then it follows from the definition of $b(\Omega)$, compare \eqref{box4}, that $ \lim_{j\rightarrow\infty}b_j(\Omega) 2^{-js} = 0 $. This implies %and \eqref{box10} imply
that there exists a constant $c$ such that   $ b_j(\Omega) \le c 2^{js}$. 
Thus $\delta > \frac{s}{\tn(p_1,p_2)}$ implies that $( 2^{-j\delta} b_j(\Omega)^{\frac{1}{\tn(p_1,p_2)}} )_{j\in\no}\in \ell_{\tn(q_1,q_2)}$, with $(2^{-j\delta} b_j(\Omega)^{\frac{1}{\tn(p_1,p_2)}}) \in c_0$ in case of $\tn(q_1,q_2) = \infty$. In view of Step~1, in particular \eqref{nuc-qb-0}, this concludes the proof of the sufficiency of \eqref{box77} for the nuclearity of $\id^B_\Omega$ in case (ii).

%Conversely let $\id_\Omega$ be nuclear and $\frac{1}{\tn(p_1,p_2)} = 0$, that is $p_1=1$ and $p_2=\infty$. Then  $\| 2^{-j\delta} N_j^{\frac{1}{\tn(p_1,p_2)}} | \ell_{\tn(q_1,q_2)}\|  = \| 2^{-j\delta} | \ell_{\tn(q_1,q_2)}\| < \infty$ holds if, and only if, $\delta > 0$ - modification $c_0$ if $\tn(q_1,q_2) = \infty$.

Now let $\id^B_\Omega$ be nuclear. Then \eqref{nuc-qb-0} implies that % and  and $\frac{1}{\tn(p_1,p_2)} > 0$. Then from $\| 2^{-j\delta} N_j^{\frac{1}{\tn(p_1,p_2)}} | \ell_{\tn(q_1,q_2)}\| < \infty$ follows  $2^{-j\delta} N_j^{\frac{1}{\tn(p_1,p_2)}} \longrightarrow 0$ respectively
$2^{-j\delta} b_j(\Omega)^{\frac{1}{\tn(p_1,p_2)}} \rightarrow 0$, which for any $s\in\real$ is equivalent to saying that 
\[ 
 2^{-j(\delta-{\frac{s}{\tn(p_1,p_2)}})}\big( 2^{-js}b_j(\Omega)\big)^{{\frac{1}{\tn(p_1,p_2)}}}\longrightarrow 0 \, .
\]
Assume now $s < b(\Omega)$, then $ \limsup\limits_{j\rightarrow\infty}b_j(\Omega) 2^{-js} = \infty $ and therefore $2^{-j(\delta-{\frac{s}{\tn(p_1,p_2)}})} \rightarrow 0 $. Thus  $\delta > \frac{s}{\tn(p_1,p_2)}$ for any $ s < b(\Omega)$ hence $\delta \ge \frac{b(\Omega)}{\tn(p_1,p_2)}$, and the proof of (ii) is complete.\\

{\em Step 3.}~ Finally we deal with the case $b(\Omega)=\infty$ and know by Step~1, that $p_1=1$, $p_2=\infty$, $s_1-s_2>d$ is sufficient for the nuclearity of $\id^B_\Omega$. So it remains to show that the condition $p_1=1$ and $p_2=\infty$ is also necessary for the nuclearity of $\id^B_\Omega$. Note that 
the necessity of the condition $\delta = s_1-s_2 - d (\frac{1}{p_1} - \frac{1}{p_2})  >0$ follows from Proposition \ref{emb2}  since otherwise the embedding is not compact.
% If $p_1=1$ and $p_2=\infty$, then $\frac{1}{t(p_1,p_2)}=0$  and $s_1-s_2 - d>0$ implies by Theorem \ref{nucl-seq-sp} nuclearity of the embedding.

If $b(\Omega) = \infty$, then by definition \eqref{box4},  for any $s>0$ there is an increasing sequence $j_k$ and a positive constant $c>0$ such that $c2^{j_k s} \le b_{j_k}(\Omega)$. This implies that the condition \eqref{nuc-qb-0} is fulfilled only if $\delta > 0$ and $\tn(p_1,p_2) = \infty$, that is, $p_1=1$ and $p_2 = \infty$. 
%Using the usual embeddings and Proposition \ref{coll-nuc} (iii) this shows, that in case $b(\Omega)=\infty$ for the $F$-spaces the embedding can be never nuclear. 
\end{proof}

\begin{remark}
If  $ 0 < \limsup_{j\rightarrow \infty} b_j(\Omega)2^{-jb(\Omega)} \le \infty $ we can replace in {\em Step 2} of the proof  $s$ with $s < b(\Omega)$ by $b(\Omega)$ itself and obtain that in this case  \eqref{box77} is a sufficient and necessary condition for nuclearity of the embeddings $\id^B_\Omega$. 
\end{remark}

\begin{remark}
The condition \eqref{box77} can be rewritten as follows:\\[-3.5ex]
\benu[(i)] 
\item If $1\le p_2\le p_1 \le \infty$, then it reads as $s_1-s_2 > b(\Omega) - d(\frac{1}{p_2} - \frac{1}{p_1})$ .
\item
If $1\le p_1\le p_2 \le \infty$, then it reads as   $s_1-s_2 > b(\Omega) - (b(\Omega) - d)(\frac{1}{p_1} - \frac{1}{p_2})$.
\eenu
Note that for $b(\Omega) = d$ these findings coincide with the condition \eqref{id_Omega-nuclear} from Proposition \ref{prod-id_Omega-nuc}. Moreover, the condition \eqref{box77} corresponds to \eqref{box7}, when $p^*$ is replaced by $\tn(p_1,p_2)$.
\end{remark}

\begin{remark}
Comparing the quite different behaviour in case (i) and (ii) of Theorem~\ref{nuclear-quasi}, one may also interpret it in the sense that, when the quasi-bounded domain becomes 'larger' in the sense that $b(\Omega)\to \infty$, then to achieve nuclearity in the sense of (ii) one needs to compensate $b(\Omega)$ on the right-hand side of \eqref{box77} by a larger number $\tn(p_1,p_2)$, too, which in the end means $\tn(p_1,p_2)=\infty$, hence only $p_1=1$, $p_2=\infty$. This represents the `smallest' source space and `largest' target space (with an appropriate interpretation in the context of Sobolev embeddings) which is possible. 
\end{remark}

We finally deal with the $F$-case and observe some new phenomenon. Recall that for the compactness result in Proposition~\ref{emb2} as well as for the compactness and nuclearity results for spaces on bounded Lipschitz domains, Propositions~\ref{prop-spaces-dom} and \ref{prod-id_Omega-nuc}, respectively, no difference between $B$- and $F$-spaces appeared. This is now different to some extent.

\begin{corollary}
Let $\Omega$ be a uniformly E-porous quasi-bounded domain in $\R^d$ and  let $s_1 > s_2$, $1\leq p_i<\infty$, $1\leq q_i\leq\infty$, $i=1,2$. %  (with $p_i<\infty$ in the $F$-case), $i=1,2$. 
Then the embedding
\begin{equation}\label{id_Omega-F}
\id^F_\Omega : \bar{F}^{s_1}_{p_1,q_1}(\Omega) \,\hookrightarrow\, \bar{F}^{s_2}_{p_2,q_2}(\Omega)
\end{equation}
is nuclear if $b(\Omega)<\infty$, and \eqref{box77} is satisfied.

Conversely, if $\id^F_\Omega$ is nuclear, then $b(\Omega)<\infty$, and 
$s_1-s_2 - d(\frac{1}{p_1} - \frac{1}{p_2})\ge \frac{b(\Omega)}{\tn(p_1,p_2)}$.

In particular, if $b(\Omega)=\infty$, then $\id^F_\Omega$ is never nuclear.
\end{corollary}

\begin{proof}
  Note first, that in case of $p_i<\infty$, $i=1,2$, that is, when all spaces involved are properly defined, then $\id_\Omega^F$ given by \eqref{id_Omega-F} is nuclear if, and only if, $\id_\Omega^B$ given by \eqref{boxemb11} is nuclear (subject to appropriately adapted fine indices $q_1,q_2$). This is due to the embedding \eqref{B-F-B} (adapted to spaces on domains) and the independence of all the conditions in Theorem~\ref{nuclear-quasi} on the fine indices $q_i$, $i=1,2$. Hence, if $b(\Omega)=\infty$ and $\id_\Omega^F$ given by \eqref{id_Omega-F} was nuclear, then in view of \eqref{B-F-B},
  \[
  \id_\Omega^B: \bar{B}^{s_1}_{p_1,\min(p_1,q_1)}(\Omega) \hookrightarrow \bar{B}^{s_2}_{p_2,\max(p_2,q_2)}(\Omega)
  \]
  is nuclear, which by Theorem~\ref{nuclear-quasi}(i) implies, in particular, $p_2=\infty$, which is not admitted.
  Assume now $b(\Omega)<\infty$, then by an argument similar to the above one, we obtain that \eqref{box77} is sufficient for the nuclearity of $\id_\Omega^F$, and $s_1-s_2 - d(\frac{1}{p_1} - \frac{1}{p_2})\ge \frac{b(\Omega)}{\tn(p_1,p_2)}$ necessary. Recall that $p_2<\infty$ implies $\tn(p_1,p_2)<\infty$.
\end{proof}

%\section*{Acknowledgements}
%The authors are deeply indebted to  the reviewers of the first version of this manuscript for  their careful reading and the valuable hints which helped us very much to improve the paper. 

%The first and third author were partially supported by the German Research Foundation (DFG), Grant no. Ha 2794/8-1. The third author  was also supported by National Science Center, Poland,  Grant No. 2013/10/A/ST1/00091.

\bigskip\bigskip~

{\small
\noindent%\begin{minipage}[t]{0.3\textwidth}
Dorothee D. Haroske\\
Institute of Mathematics \\
Friedrich Schiller University Jena\\
07737 Jena\\
Germany\\
%\vfill
{\tt dorothee.haroske@uni-jena.de}\\[4ex]
%\end{minipage}\hfill\begin{minipage}[t]{0.3\textwidth}
Hans-Gerd Leopold\\
Institute of Mathematics \\
Friedrich Schiller University Jena\\
07737 Jena\\
Germany\\
%\vfill
{\tt hans-gerd.leopold@uni-jena.de}\\[4ex]
%\end{minipage}\hfill\begin{minipage}[t]{0.3\textwidth}
Leszek Skrzypczak\\
Faculty of Mathematics \& Computer Science\\
Adam  Mickiewicz University\\
ul. Uniwersytetu Pozna\'nskiego 4\\
61-614 Pozna\'n\\
Poland\\
{\tt lskrzyp@amu.edu.pl}
%\end{minipage}\\[0ex]
}

%%%%%%%%%%%%%%%%%%%5
%%%%%%%%%%%%%%%%%%%%%%%%%%
%%%%%%%%%%%%%%%%%%%%%%%%
%%%%%%%%%%%%%%%%%%%%%%%
\end{document}